\documentclass[12pt,a4paper,oneside]{article}
\usepackage{fancyhdr}
\usepackage[latin1]{inputenc}
\usepackage{amsmath,amsfonts,amssymb,amsthm}
\usepackage{mathrsfs} 
\usepackage{commath} 
\usepackage[left=2cm,right=2cm,top=1.5cm,bottom=1cm,includeheadfoot]{geometry}
\usepackage{bm}
\usepackage[numbers]{natbib}
\usepackage[colorlinks=true, linkcolor=blue,citecolor=green,urlcolor=red]{hyperref}
\usepackage{enumitem}
\newcommand{\R}{\mathbb{R}}
\newcommand{\N}{\mathbb{N}}	
\allowdisplaybreaks
\newtheorem{defi}{Definition}
\newtheorem{thm}[defi]{Theorem}
\newtheorem*{thm*}{Theorem}
\newtheorem{lem}[defi]{Lemma}
\newtheorem{rem}[defi]{Remark}

\newtheorem{prop}[defi]{Proposition}

\begin{document}
\title{On the well-posedness of a class of McKean Feynman-Kac equations}
\author{
{\sc Jonas LIEBER}
\thanks{ENSTA Paris, Institut Polytechnique de Paris,
Unit\'e de Math\'ematiques Appliqu\'ees (UMA) 
and Humboldt-Universit\"at zu Berlin.
 E-mail:{ \tt lieber@uchicago.edu}} 
{\sc,}\ {\sc Nadia OUDJANE}
\thanks{EDF R\&D,   and FiME (Laboratoire de Finance des March\'es de l'Energie
(Dauphine, CREST,  EDF R\&D) www.fime-lab.org). 
E-mail:{\tt  
nadia.oudjane@edf.fr}}
\ {\sc and}\ {\sc Francesco RUSSO} 
\thanks{ENSTA Paris, Institut Polytechnique de Paris,
Unit\'e de Math\'ematiques Appliqu\'ees (UMA). 
 E-mail: {\tt  francesco.russo@ensta-paris.fr}}.
}

\date{October 9th 2019}

\maketitle
\begin{abstract}
\noindent We analyze the well-posedness of a so called McKean
  Feynman-Kac Equation (MFKE), which is a McKean type equation with
a Feynman-Kac perturbation.
We provide  in particular  weak and strong existence conditions as well as 
pathwise uniqueness conditions without  strong regularity assumptions on the
 coefficients.
One major tool to establish this result is a representation theorem relating
 the solutions of MFKE to  the solutions of
a nonconservative semilinear parabolic Partial Differential Equation (PDE).
\end{abstract}

\medskip\noindent {\bf Key words and phrases.}  
McKean Stochastic Differental Equations; Semilinear Partial Differential Equations; McKean Feynman-Kac equation; Probabilistic representation of
 PDEs.

\medskip\noindent  {\bf 2010  AMS-classification}: 60H10; 60H30; 60J60; 
35K58; 35C99.

\oddsidemargin -0.5cm 
\textwidth 6.5in 
\textheight 23cm
\topmargin -2cm 

\maketitle

\section{Introduction}

We discuss in this paper a mean-field type equation of the form

\begin{equation} 
\label{eq:McKean_representation}
\begin{cases}
Y_t=Y_0+\int \limits_0^t \Phi(s,Y_s) \dif W_s+\int\limits_0^t [b_0(s,Y_s)+ b(s,Y_s,u(s,Y_s)) ]\dif s  \quad \forall t \in [0,T],  \\
Y_0 \sim \mathbf{u}_0, \ \textrm{a Borel probability measure on}\ \R^d \\ 
\int\limits_{\mathbb{R}^d} \varphi (x) u(t,x)\dif x = \mathbb{E}\left[\varphi (Y_t) 
\exp\left(\int\limits_0^t \Lambda(s,Y_s,u(s,Y_s))\dif s \right)\right] \quad \forall \varphi \in C_b(\mathbb{R}^d), \forall t \in [0,T], 
\end{cases}
\end{equation}
which we will call \textit{McKean-Feynman-Kac Equation} (MFKE). A solution is given by a couple $(Y_s,u(s,\cdot))_{0\leq s\leq T}$.
 We refer to the first line of \eqref{eq:McKean_representation} as 
{\it Stochastic Differential Equation (SDE)} and to the third line as {\it linking equation}. The denomination McKean is due to the dependence of the drift coefficient in the SDE not only on time and the position of the 
process $Y$ but also on 
a function $u$ which, via the linking equation, is related to the distribution of the process.
In particular, when $\Lambda=0$ in the above equation,  $u(t,\cdot)$ 
coincides with the density of the marginal distribution $\mathcal{L}(Y_t)$. 
In this case, MFKE reduces to a  McKean Stochastic Differential Equation (MSDE), which is in general an SDE whose coefficients depend, not only on time
and position but also on $\mathcal{L}(Y_t)$. 
However, in this paper, we emphasize that the drift exhibits a \textit{pointwise} dependence on $u$;
in particular this dependence is not continuous with respect to the Wasserstein metric, as opposed to the {\it traditional setting} considered in most of the contributions in the literature. 
 Moreover the drift $b$ is possibly irregular in $x$.
\\ \\
An interesting feature of MSDEs is that the law of the  process $Y$ can often be characterized as the limiting empirical distribution of a large number  of interacting particles, whose dynamics are described by a coupled system of
classical SDEs. When the number 
of particles grows to infinity, the 
given particles behave closely to a system of 
independent copies of
$Y$. This constitutes the so called {\it propagation of chaos} phenomenon, already observed in the literature for the case of Lipschitz dependence, 
with respect to the Wasserstein metric,
 see e.g. \cite{kac, Mckeana,  Mckean, sznitman, MeleaRoel}.
\\ \\
A second important property of many MSDEs is their close relation 
 to nonlinear PDEs. 
In the present paper we propose to relate MFKE~\eqref{eq:McKean_representation} to a semilinear PDE of the form
\begin{equation}
\label{eq:PDE_original}
\left \{
\begin{array}{l}
\partial_t u = L_t^* u 
- \text{div} \left(b(t,x,u) u \right) +\Lambda(t,x,u) u\ , \quad \textrm{for any}\  t\in [0,T]\ ,\\
u(0,dx) = \mathbf{u}_0(dx),
\end{array}
\right .
\end{equation}
where $\mathbf{u_0}$ is a Borel probability measure and $L^*$
is the non-degenerate second-order linear partial differential operator such that $a=\Phi\Phi'$ and for all $t\in [0,T]$ and all $\varphi \in C_0^{\infty}(\mathbb{R}^d)$,
\begin{align} \label{eq:formal_adjoint_star}
L^*_t(\varphi)(x)=
 \frac{1}{2} \sum\limits_{i,j=1}^d \partial_{ij}^2 (a_{ij}(t,\cdot) \varphi (\cdot))(x) - \sum\limits_{j=1}^d \partial_j (b_{0,j}(t,\cdot)\varphi(\cdot))(x).
\end{align}
When $\Lambda=0$,  the above equation is a {\it non-linear} Fokker-Planck
equation, it is conservative and it is known that,
 under mild assumptions, it describes the dynamics of the marginal probability densities, $u(t,\cdot)$, of the process $Y$. 
This correspondence, between the marginal laws of a diffusion process and a
 Fokker-Planck type PDE, constitutes a representation property according 
to which a stochastic object characterizes a deterministic one and vice versa. Such representation results have extensive interesting applications. In physics, biology
or economics, it is a way to relate a microscopic model involving interacting particles to a macroscopic model involving the dynamics of some representative quantities.
 Numerically, this correspondence motivates Monte Carlo approximation schemes for PDEs. 
In particular, \cite{bossytalay1} has  contributed to develop stochastic
 particle methods in the
 spirit of McKean to provide original numerical schemes approaching a PDE
 related to 
Burgers equation providing also the rate of convergence. \\ \\
The idea of generalizing MSDEs to MFKEs was originally introduced in the sequence of papers~\cite{LOR1,LOR2,LOR4}, 
with an earlier contribution in \cite{BRR3}, where
$\Lambda(t,x,u) = \xi_t(x),$ $\xi$ being the sample
of a Gaussian noise random field, white in time and regular in space. 
\cite{LOR1} and \cite{LOR2} studied a mollified version of
\eqref{eq:PDE_original}, whose probabilistic representation 
 falls into the Wasserstein continuous
  traditional setting mentioned above. 
The underlying motivation consisted precisely in extending, to 
fairly general non-conservative PDEs, the probabilistic representation of nonlinear Fokker-Planck equations which appears when $\Lambda = 0$.
An interesting aspect of this strategy is that it is potentially
 able to represent an extended class of second order nonlinear PDEs.
 Allowing $\Lambda\neq 0$ encompasses the case of Burgers-Huxley or Burgers-Fisher equations which are of great importance to represent nonlinear phenomena in various fields such as biology~\cite{aronsonb,murray}, physiology~\cite{keener} and physics~\cite{wang}. 
These equations have the particular interest to describe the interaction between the reaction mechanisms, convection effect, and diffusion transport. \\ \\
To highlight the contribution of this paper, it is important to consider carefully the two major 
features differentiating the
MFKE~\eqref{eq:McKean_representation} from the traditional setting of MSDEs. \\
To recover the traditional setting one has to do the following.
\begin{enumerate}
\item
First, one has to put $\Lambda=0$ in the third line equation of~\eqref{eq:McKean_representation}. Then $u(t,\cdot)$ is explicitly given by the third line equation 
of~\eqref{eq:McKean_representation} and reduces to the density of 
the marginal distribution, $\mathcal{L}(Y_t)$. 
When $\Lambda \neq 0$, the relation between $u(t,\cdot)$ and the process
$Y$ is more complex.  
Indeed, not only does $\Lambda$ embed an additional nonlinearity with respect to $u$, but it also involves the whole past trajectory $(Y_s)_{0\leq s\leq t}$ of the process $Y$.  
\item Secondly, one has to replace
the pointwise dependence $b(s,Y_s,u(s,Y_s))$ in equation~\eqref{eq:McKean_representation} with a mollified dependence $b(s,Y_s,\int_{\R^d} K(Y_s,y)u(s,y)dy)$, where the dependence with respect to $u(s, \cdot)$ 
 is Wasserstein  continuous.
\end{enumerate}
Technically, in the case $\Lambda = 0$, to prove well-posedness of~\eqref{eq:McKean_representation} 
in the traditional setting, one may rely on a fixed point argument in the 
space of trajectories under the Wasserstein metric. Following the spirit of~\cite{sznitman}, this approach was carried out  in the general case in which
 $\Phi$ also shows a Wasserstein continuous dependence on $u$ in~\cite{LOR1}.
As already mentioned, the case where the coefficients depend pointwisely on $u$ is far more singular since  the dependence of the coefficients on the  law of $Y$ is no  more continuous 
with respect to the Wasserstein metrics. In this context, well-posedness results rely on analytical methods and require in general specific 
smoothness assumptions on both the coefficients and the initial condition.
 One important contribution in this direction is reported in~\cite{JourMeleard}, where strong existence and pathwise uniqueness are established when  
the diffusion coefficient $\Phi$  and the drift $g=b+b_0$ exhibit pointwise dependence on $u$ but are assumed to satisfy strong smoothness assumptions.
In this case, the solution $u$ is a classical solution of the PDE.
The specific case where the drift vanishes and the diffusion coefficient $\Phi(u(t,Y_t))$ has a pointwise dependence on the law density $u(t,\cdot)$ of $Y_t$
  has been more particularly studied in~\cite{Ben_Vallois} for classical 
porous media type equations and
\cite{BRR1,BRR2,BCR2,BCR3} who obtain well-posedness results for 
measurable and possibly singular  functions $\Phi$.
In that case the solution  $u$ of the associated PDE \eqref{eq:PDE_original},
is understood in the sense of distributions.
\\
\\
Our analysis of the well-posedness of~\eqref{eq:McKean_representation} is based on a
 different  approach. We rely on the notion of mild 
  solutions to
 Partial Differential Equations (PDEs) involving a reference semi-group
 on which the solution is built.
This is here possible since $\Phi$ does not depend on $u$, so we can allow
less regularity on the drift $g = b + b_0 $ at least with respect to time and space.
 To the best our
 knowledge, the first attempt to use this approach to analyze  McKean type SDEs 
well-posedness is reported 
in~\cite{LOR3}. 
The authors considered the case of a classical SDE
 ($b=0$), with a more singular dependence of $\Lambda$ with respect to
 $u$, since $\Lambda$ could also depend on $\nabla u$. 
In the present paper, due to the presence of 
$ \text{div}(b(t,y,u)u)$ in the PDE, 
the semi-group approach needs to be adapted. 
  An integration by parts technique takes advantage of the regularity
 of the semi-group while allowing to relax the regularity assumptions on $b$. 
Consequently, we establish the well-posedness of
 \eqref{eq:McKean_representation},
when $\Lambda$ and $b$ are only required to be bounded measurable in time and in space and Lipschitz with respect to the third variable.
We introduce in particular the cases where strong (resp.) weak
solutions appear.\\ \\
\noindent The paper is organized as follows. After this introduction,
 we clarify in Section \ref{SPreliminaries}
the notations and assumptions  
under which we work and the basic notions of weak and mild solutions
of \eqref{eq:PDE_original} in Section \ref{S3}. 
In Section \ref{SMain}, we state the main results with related proofs.
Theorem \ref{thm:representation} provides the equivalence between
solutions of the MFKE \eqref{eq:McKean_representation} 
and the PDE  \eqref{eq:PDE_original}.
Theorem \ref{thm:final} provides sufficient conditions 
for existence and uniqueness in law,
as well as strong existence and pathwise uniqueness,
 for the 
MFKE \eqref{eq:McKean_representation}.
The rest of the paper is devoted to more technical results used
to prove Theorems \ref{thm:representation} and \ref{thm:final}.
In Section \ref{S3bis} we show the equivalence between 
the notions of weak and mild solutions for \eqref{eq:PDE_original}.
  Theorem \ref{thm:e_and_u} is the key result of Section \ref{S4} 
and states existence
and uniqueness of mild solutions for \eqref{eq:PDE_original}.
Finally Proposition \ref{prop:uniqueness_measure_mild} 
 in Section \ref{subsec_equiv_3}
concerns the uniqueness of the measure-mild solution of the linear
PDE \eqref{eq:linearized_PDE}.
It is indeed the crucial tool for proving the existence of a solution 
to \eqref{eq:McKean_representation}.

\section{Basic assumptions}
\label{SPreliminaries}
\subsection{Notations}

 For a matrix $A$,
 $A^t$ denotes its transpose. 
$M_f$ is the space of signed finite measures on the Borel algebra $\mathcal{B}^d$ of $\R^d$. We equip $M_f$ with the total variation norm $\left\Vert \cdot \right\Vert_{TV}$. 
For $d\in \N^*$ and a function $f=f(t,x,\ldots)$ with $t\in [0,T]$ and $x\in \R^d$, we write
 $\partial_t f=\frac{\partial f}{\partial t}$ and 
$\partial_{x,k} f:=\frac{\partial f}{\partial x_k}$. If there is no ambiguity about $x$, we sometimes simply write $\partial_k f:=\partial_{x,k}f$. Unless we explicitly require regularity of $f$, we will interpret derivatives as distributional derivatives. 
\\ \\
Let $E$ be either $\R^d$ or $[0,T]\times \R^d$, $\vert \cdot \vert$ 
denotes the Euclidean norm of $E$. $C_0(E)$ is the set of real-valued continuous functions with compact support on $E$. $C_b(E)$  is the set of real-valued continuous bounded functions on $E$. $C_0^{\infty}(E)$ is the set of real-valued smooth functions with compact support on $E$. For $p\in [1,\infty]$, we write $L^p(E)$ for measurable real-valued functions on $E$ with finite $\left\Vert \cdot \right\Vert_{L^p}$ norm.
and $L_{\text{loc}}^p(E)$ for the locally integrable real-valued functions on $E$.
$\mathscr{M}_f(E)$ denotes the space of finite Borel signed measure in $E$.
 If $E=\R^d$, we write $C_0$ instead of $C_0(\R^d)$ and similarly $C_b$, $C_0^{\infty}$ 
and $L^p$ for $C_b(\R^d)$, $C_0^{\infty}(\R^d)$,
 and $L^p(\R^d)$ respectively.  For $k\in \N^*$, we also set $C_b^k$ to be the set of all bounded functions from $\R^d$ to $\R$ which are bounded and have continuous and bounded partial derivatives up to the $k$-th order. We then define
$C_b^{0,k}([0,T] \times \mathbb{R}^d):=\left\{ f\in C_b([0,T]\times \R^d) \ \middle| \ f(t,\cdot)\in C_b^k \ \forall  t\in [0,T] \right\}$. \\ \\ 
Finally, we denote the set of symmetric and positive semi-definite matrices in $\R^{d \times d}$ by $\mathcal{S}^d$. 
\subsection{Assumptions}   

In the whole paper we consider a matrix 
$a=(a_{ij})_{i,j=1}^d:[0,T]\times \R^d \rightarrow \mathcal{S}^d$ 
such that $a = \Phi \Phi^t$ with
 $\Phi:[0,T]\times \mathbb{R}^d \rightarrow \R^{d\times d}$ as given.
We base our discussion on the following assumptions. 
\begin{enumerate}[label=\textbf{A.\arabic*}]
\item \label{A1} The matrix $a=(a_{ij})_{i,j=1}^d:[0,T]\times \R^d \rightarrow \mathcal{S}^d$ is bounded and measurable. $b_0:[0,T]\times \mathbb{R}^d \rightarrow \R$ is bounded and measurable.
\item \label{A2} The matrix $a=(a_{ij})_{i,j=1}^d$ 
is uniformly non-degenerate, i.e. there exists a constant $\mu >0$ such that for all $(t,x)\in [0,T]\times \mathbb{R}^d$ and all $\xi \in \mathbb{R}^d$ we have 
\begin{align}\label{eq:non_degenerate}
\xi^t a(t,x) \xi = \sum\limits_{i,j=1}^d a_{ij}(t,x) \xi_i \xi_j \geq \mu \left\vert \xi \right\vert^2.
\end{align}
\end{enumerate}

\begin{enumerate}[label=\textbf{A.\arabic*}]\setcounter{enumi}{2}
\item \label{A3} For all $x\in \R^d$ 
\begin{equation} \label{conda}
\lim_{y\rightarrow x} \sup_{0\leq s\leq T} |a(s,y)-a(s,x)|=0.
\end{equation}
 \end{enumerate}
\begin{rem} \label{R128}
If $a$ is continuous, then \eqref{conda} is verified.
\end{rem}
These assumptions will come into play in Section \ref{subsec_equiv_2} when we discuss the (weak) existence and uniqueness in law of the SDE in the first line of \eqref{eq:McKean_representation} for a fixed $u$.
\noindent This will allow to get weak existence and uniqueness in law for
our MFKE. If we substitute 
assumptions \ref{A3} by the following assumption\footnote{The assumption cited here is a simplified version
 of a weaker set of assumptions. In fact, we rely here only on the assumptions of Theorem 1 in \cite{veretennikov1981strong}.}, we will get strong existence and pathwise uniqueness for MFKE. 
\begin{enumerate}[label=\textbf{B.\arabic*}]\setcounter{enumi}{2}
\item \label{B3} Assume that $a:[0,T]\times \R^d \rightarrow \mathcal{S}^d$ is continuous and that 
  $\Phi:[0,T]\times \mathbb{R}^d \rightarrow \R^{d\times d}$ which is  Lipschitz 
continuous in space, uniformly in time i.e. there exists a $L_{\Phi}>0$ such that for all $t\in [0,T]$ and 
$x,y\in \R^d$ we have 
\begin{align*}
\left\vert \Phi(t,x)-\Phi(t,y) \right\vert \leq L_{\Phi} \left\vert x-y \right\vert.
\end{align*}

\end{enumerate}
We recall that $L_t^*$ was defined in
\eqref{eq:formal_adjoint_star}.
We define
\begin{align} \label{eq:formal_adjoint}
L_t(\varphi)(x)=\frac{1}{2} \sum\limits_{i,j=1}^d  a_{ij}(t,x) \partial_{ij}^2 \varphi (x) + \sum\limits_{j=1}^d b_{0,j}(t,x) \partial_j  \varphi(x),
\end{align} 
for all $t\in [0,T]$ and all $\varphi \in C_0^{\infty}(\mathbb{R}^d)$. Now we can introduce the  Fokker-Planck equation, i.e. for $0 \le s < T$,
\begin{align}\label{eq:FokkerPlanck}
\begin{cases} \partial_t \nu(t,x) = &L_t^* \nu(t,x)   \quad \forall (t,x) \in (s,T]\times \R^d \\
\nu(s,\cdot)= & \nu_0,
\end{cases}
\end{align}
where $\nu_0$ is a probability measure on $\mathcal{B}^d$.  
We introduce the notion of fundamental solution, for simplicity 
under Assumption \ref{A1}.
\begin{defi} \label{D3} 
A Borel function $p:[0,T] \times \mathbb{R}^d \times [0,T]\times \mathbb{R}^d \rightarrow \R_+$ is a \textbf{fundamental solution} to \eqref{eq:FokkerPlanck},
with the convention that
$ p(s,\cdot,t, \cdot) = 0$ a.e., if $s > t$,
if  the following holds.
\begin{enumerate}
\item For every $s, x_0, t$ such that $0 \le s < t \le T,$
\begin{equation} \label{Intequal1}
 \int\limits_{\mathbb{R}^d} p(s,x_0,t,x) dx = 1.
\end{equation}
\item
 For every probability distribution $\nu_0$ on $\R^d$,  the function
\begin{align}\label{fundamentalsolution}
\nu_s(t,x):=\int\limits_{\mathbb{R}^d} p(s,x_0,t,x)\nu_0(\dif x_0)
\end{align}
is a  solution in the sense of distributions to \eqref{eq:FokkerPlanck}
i.e., for all $\varphi \in C_0^{\infty},$
\begin{align}\label{eq:rather_implicit}
  \int\limits_{\mathbb{R}^d} \varphi (x)\nu_s(t,x)\dif x -\int\limits_{\mathbb{R}^d} \varphi (x)\nu_0(\dif x)=\int\limits_s^t \int\limits_{\mathbb{R}^d} L_r\varphi (x)
 \nu_s(r,x)\dif x \dif r.  
\end{align}
\end{enumerate}
\end{defi}
\begin{rem}
\begin{enumerate}
\item
Proposition \ref{PAppReg} provides sufficient conditions for 
the existence of a fundamental solution.
\item In many examples, a fundamental solution 
in the sense of Definition \ref{D3}
is also a fundamental solution 
of $L_t^*u - \partial_t u = 0$ 
in the terminology of Friedman, see Definition in sect. 1, p.3 of \cite{friedmanEDP},
for instance under the conditions of Proposition \ref{PAppReg}. 
The details are provided in the proof.
\item 
By the validity of \eqref{Intequal1}, the expressions 
\eqref{eq:rather_implicit}
 and \eqref{fundamentalsolution} 
make sense. We also say that $p$ is a {\it Markov fundamental solution}.
\end{enumerate}
\end{rem}
\noindent Now, we are in the position to impose our assumptions on the fundamental solution.
\begin{enumerate}[label=\textbf{A.\arabic*}]
\setcounter{enumi}{3}
\item \label{A4} 
There exists a fundamental solution
$p$ to \eqref{eq:FokkerPlanck} with the following properties.
\begin{enumerate}
\item The first order partial derivatives of the map $x_0 \mapsto p(s,x_0,t,x)$Let us consider the family of Markov 
transition functions $P(s,x_0,t, \cdot)$
associated with $(L_t)$, see \cite{LOR3}. exist in the distributional sense.
\item For almost all $0 \leq s < t \leq T$ and $x_0,x \in \R^d$ there are constants $C_u,c_u>0$ 
such that 
\begin{align}
\label{eq:1011a}
 p(s,x_0,t,x) \leq C_u
q(s,x_0,t,x)
\end{align}
and
\begin{align}
\label{eq:1011}
\left\vert \partial_{x_0} p(s,x_0,t,x) \right\vert \leq C_u
 \frac{1}{\sqrt{t-s}} q(s,x_0,t,x)\ ,
\end{align}
where
$q(s,x_0,t,x):=\left(\frac{c_u(t-s)}{\pi}\right)^{\frac{d}{2}} e^{-c_u \frac{\vert x-x_0 \vert^2}{t-s}}$ 
is a Gaussian probability density.
\item The following {\it Chapman-Kolmogorov type equality} holds:
 for all $0\leq s \leq t\leq r$ and for almost all $x_0,y \in \R^d$ we have
\begin{equation} \label{MarkovFund}
p(s,x_0,r,y)=\int\limits_{\R^d} p(s,x_0,t,x) p(t,x,r,y )\dif x.
\end{equation}
\end{enumerate}
 \end{enumerate}
The next two assumptions concern $b$ and $\Lambda$.
\begin{enumerate}[label=\textbf{A.\arabic*}]
\setcounter{enumi}{4}
\item \label{A5} $b:[0,T]\times \mathbb{R}^d \times \mathbb{R} \rightarrow \mathbb{R}^d$ is uniformly bounded by $M_b>0$. 
Similarly, $\Lambda:[0,T]\times \mathbb{R}^d \times \mathbb{R} \rightarrow \mathbb{R}$ is uniformly bounded by 
$M_{\Lambda} >  0$. 

\item \label{A6}
\begin{enumerate}
\item $b:[0,T]\times \mathbb{R}^d \times \mathbb{R} \rightarrow \mathbb{R}^d$ is Borel, Lipschitz continuous in its third argument, uniformly in space and time, i.e. there exists an $L_b>0$ such that for all $t\in [0,T], x\in \R^d$ and  $z_1,z_2 \in \R$ one has
\begin{align*}
\left\vert b(t,x,z_1)-b(t,x,z_2) \right\vert \leq L_b \left\vert z_1- z_2 \right\vert. 
\end{align*}
\item Similarly, $\Lambda:[0,T]\times \mathbb{R}^d \times \mathbb{R} \rightarrow \mathbb{R}$ is Borel, Lipschitz continuous in its third argument, uniformly in space and time, i.e. there exists an $L_b>0$ such that for all $t\in [0,T], x\in \R^d$ and  $z_1,z_2 \in \R$ one has
\begin{align*}
\left\vert \Lambda(t,x,z_1)-\Lambda(t,x,z_2) \right\vert \leq L_{\Lambda} \left\vert z_1- z_2 \right\vert. 
\end{align*}
\item We also suppose that $(t,x) \mapsto \Lambda (t,x,0)$ and 
$(t,x) \mapsto b (t,x,0)$
are bounded. 
\end{enumerate}
\end{enumerate}
By now, we have imposed assumptions on all the terms appearing in the PDE \eqref{eq:PDE_original} as well as in MFKE~\eqref{eq:McKean_representation}. Let us consider now the initial condition of both the PDE and the MFKE.
\begin{enumerate}[label=\textbf{A.\arabic*}]
\setcounter{enumi}{6}
\item \label{A7} ${\bf u_0}$ admits a bounded density $u_0$ with respect to the Lebesgue measure.
\end{enumerate}

\noindent Except for \ref{A4}, all assumptions are straightforward to verify. Fortunately, there are results specifying manageable conditions which are sufficient for \ref{A4}. 
\begin{rem}\label{rem:Friedman_conditions} 
Proposition \ref{PAppReg} in the Appendix 
precises regularity conditions on the coefficients $(a_{ij})$ and $(b_0)_j$ which ensure that Assumption \ref{A4} is fulfilled.
\end{rem}

\subsection{Weak and mild solutions}
\label{S3}
We will begin by introducing the notion of a weak solution to \eqref{eq:PDE_original}. 
\begin{defi} \label{def_weak_sol}
 Assume \ref{A1} and $b, \Lambda$ to be locally bounded.
 A \textbf{weak solution of  PDE \eqref{eq:PDE_original}} is given by a function $u:[0,T]\times \mathbb{R}^d \rightarrow \mathbb{R}$ such that $u \in L_{\text{loc}}^1([0,T]\times \mathbb{R}^d)$ and we have, for any $\varphi \in C_0^{\infty}$,
\begin{align}
\int\limits_{\mathbb{R}^d} \varphi(x)u(t,x)\dif x = &\int\limits_{\mathbb{R}^d} \varphi(x){\bf u_0}(\dif x) + \int\limits_0^t \int\limits_{\mathbb{R}^d} u(s,x)L_s \varphi(x) \dif x \dif s \notag \\ 
 & +\sum\limits_{j=1}^d \int\limits_0^t \int\limits_{\mathbb{R}^d} \partial_j (\varphi (x)) b_j(s,x,u(s,x))u(s,x) \dif x \dif s \notag\\
& + \int\limits_0^t \int\limits_{\mathbb{R}^d} \varphi(x)  \Lambda (s,x,u(s,x)) u(s,x) \dif x\dif s. \label{eq:weak_sol_def}
\end{align}
\end{defi}
\noindent 
Contrarily to the case  when $b_0$ and $b$
vanish, but $\Phi$ may depend on $u$ as in \cite{BRR1, BRR2, BCR2},
 we do not have analytical tools at our disposal with which existence and uniqueness of weak solutions can be established. This is the main reason why we now introduce the notion of mild solutions. Assume that $p$ is a fundamental solution of~\eqref{eq:FokkerPlanck} in the sense of Assumption \ref{A4}. 
The classical and natural formulation of mild solution for \eqref{eq:PDE_original} looks as
follows:
\begin{align*}
u(t,x)=&\int_{\mathbb{R}^d} p(0,x_0,t,x)\mathbf{u_0}(\dif x_0)+\int_0^t\int_{\mathbb{R}^d} \Lambda(t,x_0,u(t,x_0))u(t,x_0)p(s,x_0,t,x)\dif x_0 \dif s\\ &-\sum\limits_{j=1}^d \int_0^t\int_{\mathbb{R}^d} \partial_{x_0,j} \left[b_j(t,x_0,u(t,x_0))u(t,x_0) \right]p(s,x_0,t,x)\dif x_0 \dif s.
\end{align*}
Note, however, that we do not assume any differentiability of $b$ in space and therefore a definition including $\partial_{x_0,j} b_j(t,x_0,u(t,x_0))$ does not make sense. This motivates to formally integrate by parts. The boundary term would disappear if we set $u\in L^1([0,T]\times \R^d)$ because $b$ is bounded and $p$ satisfies inequality \eqref{eq:1011}. By shifting the derivative from $b$ to $p$, this leads to the following definition.
	\begin{defi} \label{def_mild_sol} 
Suppose that assumptions \ref{A1}, \ref{A4} and \ref{A5} hold.
 A function $u:[0,T]\times \mathbb{R}^d \rightarrow \R$ will be called \textbf{mild solution of \eqref{eq:PDE_original}} if $u \in L^1([0,T] \times \mathbb{R}^d)$ and we have, for any $t \in [0,T]$ , that
\begin{align}
u(t,x)= &\int\limits_{\mathbb{R}^d} p(0,x_0,t,x){\bf u_0} (\dif x_0) + \int\limits_0^t \int\limits_{\mathbb{R}^d}u(s,x_0)\Lambda(s,x_0,u(s,x_0))p(s,x_0,t,x) \dif x_0 \dif s \notag \\
&+ \sum\limits_{j=1}^d \int\limits_0^t \int\limits_{\mathbb{R}^d}u(s,x_0) b_j(s,x_0,u(s,x_0)) \partial_{x_0,j} p(s,x_0,t,x) \dif x_0 \dif s \label{eq:mild_sol_def}.
\end{align} 
\end{defi}
\noindent We observe that, whenever $u \in L^1([0,T] \times \R^d)$ and assumptions  \ref{A1}, \ref{A4} and \ref{A5} hold, the right-hand side of  \eqref{eq:mild_sol_def} is indeed a well-defined function in $ L^1([0,T] \times \R^d)$. 

\medskip
\noindent
In the sequel we will often make use of the assumption below, which is
in fact automatically implied by Assumptions \ref{A1}, \ref{A2}, \ref{A3}.
\begin{enumerate}[label=\textbf{C}]
\item \label{C}
The PDE 
\begin{align}
\begin{cases}
\partial_t w &= L_t^* w, \\
w(0,\cdot)&= 0, \label{eq:uniqueness_pde}
\end{cases}
\end{align}
admits $w=0$ as unique weak solution 
among measure valued functions from $[0,T]$ to ${\mathcal M} ({\mathbb R^d})$, i.e. there exists a
unique measure-weak solution in the sense of  Definition~\ref{def:measure_weak_solution}.
\end{enumerate}

\begin{rem} \label{rem:discussion_uniqueness_pde}
\begin{enumerate}
\item
As we have written above, under Assumptions \ref{A1}, \ref{A2} and \ref{A3},
Assumption  \ref{C} is verified. 
Indeed by  Lemma 2.3 in \cite{figalli}, Assumption \ref{C} holds when
the martingale problem related to $(L_t)$ is well-posed.
Under Assumptions \ref{A1}, \ref{A2} and \ref{A3}
this is always the case
see e.g.  Theorem 7.2.1 in Chapter 7, Section 2 of \cite{stroock}.
\item Other conditions of validity for the uniqueness of weak solutions to
 \eqref{eq:uniqueness_pde} are discussed in the literature.
 Recent results include Theorem 1 in \cite{bogachevSFB}, Theorem 1.1 in \cite{roeckner2010} and Theorem 3.1 in \cite{BCR2}. 

\end{enumerate}
\end{rem}

\section{Main results and strategy of the proofs}
\label{SMain}

\subsection{Well-posedness for McKean Feynman-Kac equation}

Throughout this section, we impose Assumption \ref{A1}.
 In particular, \ref{A1} implies that we can write $a(t,x)=\Phi(t,x)\Phi(t,x)^t$ for some bounded $\Phi$ which we fix.
 Given a filtered probability space 
$(\Omega, \mathcal{F},(\mathcal{F}_t)_{t\in [0,T]},\mathbb{P})$ equipped with a $d$-dimensional
$(\mathcal{F}_t)_{t \in [0,T]}$ Brownian motion $(W_t)_{t\in [0,T]}$ and an ${\mathcal F}_0$-measurable random variable $Y_0 \sim {\bf u_0}$, we say  that
 a couple $(Y,u)$ is a solution to~\eqref{eq:McKean_representation} if the following conditions hold.
\begin{enumerate}
\item $Y$ is an  $(\mathcal{F}_t)_{t \geq 0}$-adapted process and $u:[0,T] \times \R^d \rightarrow \R $, such that $(Y,u)$ verifies~\eqref{eq:McKean_representation};
\item $(t,x) \mapsto b(t,x,u(t,x))$ and $(t,x) \mapsto \Lambda(t,x,u(t,x))$ 
are bounded. 
\end{enumerate}
Below we introduce the notions of existence and uniqueness to \eqref{eq:McKean_representation}.
The uniqueness aspect will be defined only in the class of
pairs $(Y,u)$ such that $u$ is bounded.

\begin{defi} \label{def-strong-sol}\hfill
\begin{enumerate} 
\item We say that \eqref{eq:McKean_representation} admits \textbf{strong existence}
if for any complete filtered probability space 
$(\Omega, \mathcal{F},(\mathcal{F}_t)_{t\in [0,T]},\mathbb{P})$ equipped with a $d$-dimensional
$(\mathcal{F}_t)_{t \in [0,T]}$ Brownian motion $(W_t)_{t\in [0,T]}$ and an
 ${\mathcal F}_0$-measurable random variable $Y_0 \sim {\bf u_0}$, there is a couple $(Y,u)$ such that $Y$ is an  $(\mathcal{F}_t)_{t \geq 0}$-adapted process, $u:[0,T] \times \R^d \rightarrow \R $ and $(Y,u)$ is a solution to \eqref{eq:McKean_representation}. 
\item We say that \eqref{eq:McKean_representation} admits \textbf{pathwise uniqueness} if for any complete filtered probability space 
$(\Omega, \mathcal{F},(\mathcal{F}_t)_{t\in [0,T]},\mathbb{P})$ equipped with a $d$-dimensional
$(\mathcal{F}_t)_{t \in [0,T]}$-Brownian motion $(W_t)_{t\in [0,T]}$, an $\mathcal{F}_0$-random variable $Y_0\sim {\bf u_0}$, the following holds.
For any given two  pairs $(Y^1,u^1)$ and $(Y^2,u^2)$ of solutions
 to \eqref{eq:McKean_representation} such that 
$u^1, u^2$ are bounded and
$Y^1_0 = Y^2_0$ $\mathbb{P}$-a.s. we have that $Y^1$ and $Y^2$ are indistinguishable and $u^1 = u^2$.
\end{enumerate}
\end{defi}
\begin{defi} \label{def-weak-sol}\hfill 
\begin{enumerate}
\item
We say that \eqref{eq:McKean_representation}  admits \textbf{existence in law} if 
there is a complete filtered probability space 
$(\Omega, \mathcal{F},(\mathcal{F}_t)_{t\in [0,T]},\mathbb{P})$ equipped with a $d$-dimensional
$(\mathcal{F}_t)_{t \in [0,T]}$-Brownian motion $(W_t)_{t\in [0,T]}$, 
 a  pair $(Y,u)$  solution of \eqref{eq:McKean_representation}, where 
$Y$ is an $(\mathcal{F}_t)_{t \in [0,T]}$-adapted process
 and  $u$ is a real valued function defined on $[0,T] \times \R^d$. 
\item 
We say that  \eqref{eq:McKean_representation}  admits \textbf{uniqueness in law}, if the following holds.
For any two solutions, $(Y,u)$ and $(\tilde{Y},\tilde{u}$) of \eqref{eq:McKean_representation}, defined on complete filtered probability spaces $(\Omega, \mathcal{F},(\mathcal{F}_t)_{t\in [0,T]},\mathbb{P})$ and $(\tilde{\Omega}, \tilde{\mathcal{F}},(\tilde{\mathcal{F}}_t)_{t\in [0,T]},\tilde{\mathbb{P}})$, 
respectively, which are such that $u$ and $\tilde u$ are bounded, then $u=\tilde{u}$ and $Y$ and $\tilde{Y}$ have the same law.
\end{enumerate}
\end{defi}

\subsection{A relaxation of Assumption \ref{A5}}.
While Assumption \ref{A5} is convenient, 
it is not verified in many interesting situations.
However, it can be replaced with the following assumption.
\begin{enumerate}[label=\textbf{C5}]
\item \label{C5}
There exists a mild solution  $u \in L^1([0,T]\times \mathbb{R}^d)$
of \eqref{eq:PDE_original},
such that $(t,x) \mapsto b(t,x,u(t,x))$ and 
$(t,x) \mapsto \Lambda(t,x,u(t,x))$ 
are bounded.
\end{enumerate}
\begin{rem} \label{RGeneral}
Under Assumption \ref{A6}, if
$u$ is a bounded mild solution of  \eqref{eq:PDE_original}
then Assumption \ref{C5} is verified.
\end{rem}
\begin{rem}\label{rem:Burgers_apriori}
For example, consider Burgers' equation
\begin{align}
\begin{cases}
\partial_t u(t,x)&=\frac{\nu}{2}\partial_{xx}^2 u(t,x)-u(t,x)\partial_x u(t,x),\qquad \forall (t,x)\in [0,T]\times \mathbb{R},\\
u(0,\cdot)&=u_0,
\end{cases}
\label{eq:Burgers_equation}
\end{align}
where the constant $\nu>0$ and $u_0$ is a bounded probability density with respect to the Lebesgue measure. In our framework, the representation would be
 $b(t,x,z) =\frac{1}{2} z$, $\Lambda(t,x,z)=0$, $b_0(t,x)=0$ and $\Phi(t,x)=\sqrt{\nu}$. As the reader may easily verify using Proposition \ref{PAppReg}, all assumptions except for \ref{A5} are satisfied. Assumption \ref{A5} is violated because $ z \mapsto \frac{1}{2} z$ 
 is only locally bounded.
However Assumption \ref{C5} is verified.
Indeed there exists a bounded classical solution $u$.
In fact for instance \cite{DelarueMenozzi} states that
 \eqref{eq:Burgers_equation} admits a classical solution 
given by
\begin{eqnarray}
\label{eq:ExpliBurg}
u(t,x) = \frac{{\mathbb E}[u_0(x+ \nu B_t)e^{-\frac{U_0(x+ \nu B_t)}{\nu^2}}]}
{{\mathbb E}[e^{-\frac{U_0(x+\nu B_t)}{\nu^2}}]}, \quad (t,x) \in [0,T] \times \R, 
\end{eqnarray}
where $B$ denotes a one-dimensional Brownian motion and $U_0$ is the cumulative distribution function associated with $u_0$. Since $u_0$ is bounded, $u$ is obviously bounded.
This is therefore a weak solution. 
Taking into account
 Assumption \ref{C} and Proposition \ref{prop:equiv_mild_weak_sol},  
it is also a mild solution. Then by Remark \ref{RGeneral}, Assumption \ref{C5}
is verified.
\end{rem}

\subsection{Main results}
In this section, we state the two main theorems of the article.
\begin{thm} \label{thm:representation}
Assume that Assumptions \ref{A1}, \ref{A2}, \ref{A3}, \ref{A4} 
 hold.
\begin{enumerate}
\item Let $(Y,u)$  
 be a solution of \eqref{eq:McKean_representation},
where $u \in L^1([0,T]\times \mathbb{R}^d)$:
in particular $(t,x) \mapsto b(t,x,u(t,x))$ and 
$(t,x) \mapsto \Lambda(t,x,u(t,x))$ 
are bounded.
Then $u$ is a weak solution of \eqref{eq:PDE_original}.
\item 
 Let 
$u \in L^1([0,T]\times \mathbb{R}^d)$
 be a weak solution of \eqref{eq:PDE_original}
 such that
  $(t,x) \mapsto b(t,x,u(t,x))$ and 
$(t,x) \mapsto \Lambda(t,x,u(t,x))$ 
are bounded. 
Then  \eqref{eq:McKean_representation} admits existence 
and uniqueness  in law.
In particular there is  a (unique in law) process $Y$ such that
 $(Y,u)$ is a solution of \eqref{eq:McKean_representation}.
\item 
If Assumption \ref{A3} is replaced with  \ref{B3}, 
 we obtain an analogous results to 
item 2. where
existence and uniqueness in law 
 is replaced by strong existence
and pathwise uniqueness.
\end{enumerate}
\end{thm}
This, together with Theorem \ref{thm:e_and_u}  will allow us to formulate 
 existence and uniqueness for the MFKE equation.
\begin{thm} \label{thm:final} 
Assume that \ref{A1},\ref{A2}, \ref{A3}, \ref{A4}, 
\ref{A6}, \ref{A7} 
hold.
\begin{enumerate}
 \item Under Assumption \ref{A5} or Assumption \ref{C5}
\eqref{eq:McKean_representation} admits existence in law.
Under Assumption \ref{A5} the constructed solution
$(Y,u)$ is such that $u$ is bounded. 
\item  
\eqref{eq:McKean_representation} 
admits uniqueness in law.
\item Suppose \ref{A3} is replaced with
 \ref{B3}.
\begin{itemize} 
\item Under Assumption \ref{A5} or \ref{C5}
 we  have strong existence 
of a solution to \eqref{eq:McKean_representation}.
\item Under Assumption \ref{A5} we have pathwise
uniqueness of a solution $(Y,u)$
to \eqref{eq:McKean_representation}. 
\end{itemize}
\end{enumerate}
\end{thm}
  \begin{rem} \label{RModifVeret}
 Concerning item 1. of the theorem above, if we substitute \ref{A3} with the hypothesis
that $\Phi$ is non-degenerate, i.e. \eqref{eq:non_degenerate}
is fulfilled replacing $a$ with $\Phi$,
we still have weak existence of a 
solution $(Y,u)$ (but not necessarily uniqueness) to \eqref{eq:McKean_representation}. Instead of appealing to 
 Lemma \ref{T127}
  for the weak existence and uniqueness in law of a solution to the SDE,
 we then appeal to
  Remark \ref{rem:e_without_u}. 
 \end{rem}

\subsection{Strategy of the proof}
In this section, we sketch the proofs to our main theorems, beginning with Theorem
\ref{thm:representation}. 

\subsubsection{Proof of Item 1. of Theorem \ref{thm:representation}}
\label{SRT}

Item 1. 
of Theorem \ref{thm:representation} follows from  Proposition~\ref{prop:strong_sol_of_sde_is_weak_sol_of_pde} below, which is obtained by a direct application of It\^o's formula.
\begin{prop} \label{prop:strong_sol_of_sde_is_weak_sol_of_pde}
Suppose the validity of Assumption \ref{A1}.
Suppose that $(Y,u)$ is a solution 
of \eqref{eq:McKean_representation}.
 Then $u$ is a weak solution of \eqref{eq:PDE_original}.
\end{prop} 
\begin{proof}[Proof \nopunct.]
Let $\varphi \in C_0^{\infty}(\mathbb{R}^d)$ be a test function.
We use stochastic integration by parts to infer that
\begin{eqnarray} 
\label{eq:proof_one_way_ito}
\varphi(Y_t)\exp\Big (\int\limits_0^t \Lambda(s,Y_s,u(s,Y_s)) \dif s \Big ) 
&=& \varphi(Y_0)\exp\left(0 \right) + \int\limits_0^t \exp \Big (\int\limits_0^s \Lambda(r,Y_r,u(r,Y_r)) \dif r \Big )  \dif \varphi(Y_s)  \nonumber \\
&&+ \int\limits_0^t \varphi(Y_s) \dif \ \exp \Big (\int\limits_0^s \Lambda(r,Y_r,u(r,Y_r)) \dif r \Big). 
\end{eqnarray}
To further develop the second term on the right-hand side, an application of It\^{o}'s formula yields
\begin{align*}
\varphi (Y_s)=&\varphi (Y_0)+\sum\limits_{i=1}^d \int\limits_0^s \partial_i \varphi (Y_r) \dif Y_r^i + \frac{1}{2} \sum\limits_{i,j=1}^d \int\limits_0^s \partial_{ij}^2 \varphi (Y_r) \dif \langle Y^i, Y^j \rangle_r \\
= &\varphi (Y_0)+\sum\limits_{i=1}^d \Big ( \sum\limits_{j=1}^d \int\limits_0^s \partial_i  \varphi (Y_r) \Phi_{ij}(r,Y_r) \dif W_r^j +\int\limits_0^s \partial_i \varphi (Y_r)(b_{0,i}(r,Y_r)+b_i(r,Y_r,u(r,Y_r)) \dif r \Big ) \\
&+ \frac{1}{2} \sum\limits_{i,j=1}^d \int\limits_0^s 
 a_{ij}(r,Y_r)
\partial_{ij}^2 \varphi (Y_r) \dif r \\
= &\varphi (Y_0)+\sum\limits_{i,j=1}^d \int\limits_0^s \partial_i  \varphi (Y_r) \Phi_{ij}(r,Y_r) \dif W_r^j + \int\limits_0^s L_r \varphi (Y_r)\dif r + \sum\limits_{i=1}^d  \int\limits_0^s \partial_i  \varphi (Y_r) b_i(r,Y_r,u(r,Y_r)) \dif r.
\end{align*}
For the third term we write
\begin{align*}
\exp\Big(\int\limits_0^s \Lambda(r,Y_r,u(r,Y_r)) \dif r \Big)= 1+\int\limits_0^s  \exp\Big (\int_0^r \Lambda(z,Y_z,u(z,Y_z))\dif z\Big ) \Lambda(r,Y_r,u(r,Y_r)) \dif r.
\end{align*}
Now, plugging all this into equation \eqref{eq:proof_one_way_ito} leaves us with
\begin{align*}
&\varphi(Y_t)\exp\Big(\int\limits_0^t \Lambda(s,Y_s,u(s,Y_s)) \dif s \Big)\\ 
=& \varphi(Y_0)\exp(0) + \sum\limits_{i,j=1}^d \int\limits_0^t \exp\Big (\int\limits_0^s \Lambda(r,Y_r,u(r,Y_r)) \dif r \Big ) \partial_i  \varphi (Y_s)\Phi_{ij}(s,Y_s) \dif W_s^j \notag \\
&+\sum\limits_{i=1}^d\int\limits_0^t \exp\Big (\int\limits_0^s \Lambda(r,Y_r,u(r,Y_r)) \dif r \Big ) \partial_i \varphi (Y_s) b_i(s,Y_s,u(s,Y_s)) \dif s\\
&+\int\limits_0^t \exp\Big (\int\limits_0^s \Lambda(r,Y_r,u(r,Y_r)) \dif r \Big ) L_s \varphi(Y_s) \dif s  \\
&+ \int\limits_0^t \varphi(Y_s)\exp \Big (\int_0^s \Lambda(z,Y_z,u(z,Y_z))\dif z \Big ) \Lambda(s,Y_s,u(s,Y_s)) \dif s, 
\end{align*}
which almost finishes the proof. Indeed, we now only have to take the expectation using
Fubini's theorem to exchange the integral with respect to time with the expectation and then apply the third line 
of \eqref{eq:McKean_representation} which gives exactly \eqref{eq:weak_sol_def}.
\end{proof}

\subsubsection{Sketch of the proof of Item 2. of Theorem \ref{thm:representation}}
To establish item 2. of Theorem \ref{thm:representation}, we proceed through the following steps.
\begin{enumerate}
\item Let $u$ be a (weak) solution to \eqref{eq:PDE_original}
such that $(t,x) \mapsto b(t,x,u(t,x))$ and 
$(t,x) \mapsto \Lambda(t,x,u(t,x))$ 
are bounded.
\item By Proposition \ref{prop:equiv_mild_weak_sol}, $u$ is also a mild solution to \eqref{eq:PDE_original}.
\item Using Stroock-Varadhan arguments, see Lemma \ref{T127},
we construct a process $Y$ satisfying (in law) the SDE
\begin{align}\label{eq:just_the_SDE}
\begin{cases} 
Y_t&=Y_0+\int\limits_0^t \Phi(s,Y_s) \dif W_s+\int\limits_0^t (b_0(s,Y_s)+ b(s,Y_s,u(s,Y_s)) \dif s  \quad \forall t \in [0,T],\\
Y_0&\sim {\bf u_0}.
\end{cases}
\end{align}
\item Let $t \ge 0$. We can then define a measure $\mu_t$ on $(\mathbb{R}^d,\mathcal{B}^d)$ by setting for any measurable and bounded $\varphi:\mathbb{R}^d\rightarrow \mathbb{R}$ 
 \begin{align}
\int\limits_{\mathbb{R}^d} \varphi(x) \mu_t( \dif x):= 
 \mathbb{E}\left[ \varphi(Y_t) \exp\left(\int_0^t \Lambda(s,Y_s,u(s,Y_s))\dif s\right) \right].
 \label{eq:linking_almost}
 \end{align}
We show that $\mu_t$ is a solution in the sense of distributions
(i.e. a measure-weak solution in the sense of Definition~\ref{def:measure_weak_solution}) of the linear equation 
\eqref{eq:linearized_PDE}
 with $\hat \Lambda(s,y) = \Lambda(s,y,u(s,y))$, 
$\hat b(s,y) = b(s,y,u(s,y))$. In fact, we basically apply integration by parts and It\^o's formula on $\varphi(Y_t)\exp\left( \int_0^t \Lambda(s,Y_s, u(s,Y_s))
\dif s \right)$ similarly to the proof of Proposition \ref{prop:strong_sol_of_sde_is_weak_sol_of_pde}.  
According to Proposition \ref{prop:equiv_mild_weak_sol_lin}, this implies 
that $\mu$ is  a measure-mild solution of \eqref{eq:linearized_PDE} in the sense of Definition \ref{def:measure_mild_solution}. 
\item On the other hand, by item 2. $\nu_t(\dif x)= u(t,x)\dif x$
is obviously also a measure-mild solution of \eqref{eq:linearized_PDE}. We thus have two measure-mild solutions of \eqref{eq:linearized_PDE}, $\mu_t(\dif x)$ and $ \nu_t(\dif x).$
\item  Proposition \ref{prop:uniqueness_measure_mild} states uniqueness 
of measure-mild solutions to \eqref{eq:linearized_PDE}. This allows us to infer
$\mu_t(\dif x)=u(t,x)\dif x$, thereby completing the proof
of item 2. of Theorem \ref{thm:representation}.
\end{enumerate}
Item 3. can be established replacing Stroock-Varadhan arguments
(see Lemma \ref{T127})  with Lemma \ref{veretennikov} 
which states strong existence of \eqref{eq:just_the_SDE}.

\bigskip

\subsubsection{Sketch of the proof of Theorem \ref{thm:final}}
Now, we outline the proof of Theorem \ref{thm:final}.
\begin{description}
\item{1.}
Let $u$ be a mild solution of \eqref{eq:PDE_original}, as stated by Assumption \ref{C5} 
or guaranteed by item 1. of Theorem  \ref{thm:e_and_u}
 if Assumption \ref{A5} holds; in that latter case $u$ is bounded.
 According to Proposition \ref{prop:equiv_mild_weak_sol}, $u$ is then a
 weak solution of \eqref{eq:PDE_original}. Hence, by item 2. of Theorem \ref{thm:representation}, 
we conclude that there is a stochastic process $Y$ such that the couple $(Y,u)$ solves \eqref{eq:McKean_representation} in law. 
\item{2.}
As for uniqueness, suppose there are two solutions $(Y_1,u_1)$ and $(Y_2,u_2)$ of \eqref{eq:McKean_representation} such that $u_1, u_2$ are bounded.
In that case 
by Assumption \ref{A6}  
$(s,x) \mapsto b(s,x, u_i(s,x)), (s,x) \mapsto \Lambda(s,x,u_i(s,x))$
are bounded.
Then, by item 1. of 
Theorem \ref{thm:representation}, both $u_1$ and $u_2$ are weak solutions to \eqref{eq:PDE_original}. Then, by Proposition \ref{prop:equiv_mild_weak_sol}, $u_1$ and $u_2$ are also mild solutions to \eqref{eq:PDE_original}. 
But by item 2. of Theorem \ref{thm:e_and_u}, mild solutions are unique, i.e. $u_1=u_2$. 
We conclude 
by Lemma \ref{T127}
that the law of $Y_1$ is the same as the law of $Y_2$.
\item{3.} The same proof works substituting item 2.  
of Theorem \ref{thm:representation} with item 3.
and Lemma  \ref{T127} with Lemma \ref{veretennikov}.
 \end{description}
In the sequel of the paper we will state  and establish
the main ingredients of these proofs.

\section{Equivalence of weak and mild solutions
of the nonlinear PDE}
\label{S3bis}

The main result of this section is 
 Proposition~\ref{prop:equiv_mild_weak_sol}, which specifies conditions under which weak and mild solutions to \eqref{eq:PDE_original} are equivalent.

\begin{prop} \label{prop:equiv_mild_weak_sol}
 Let $u$
 such that $(t,x) \mapsto b(t,x,u(t,x))$  and 
$(t,x) \mapsto \Lambda(t,x,u(t,x))$  are bounded.
Assume \ref{A1} and \ref{A4}. 
\begin{enumerate}
\item If $u$ 
  is  a mild solution of \eqref{eq:PDE_original}
then it is also a weak solution.
\item Conversely, if Assumption \ref{C} holds and
$u$ is a weak solution of \eqref{eq:PDE_original}, 
then $u$ is also a mild solution of \eqref{eq:PDE_original}.
\end{enumerate}
\end{prop}

\noindent Proposition \ref{prop:equiv_mild_weak_sol} is a key result for this work because we do not have analytical tools at our disposal to establish existence or even uniqueness of a weak solution directly. However, we can apply fixed-point theorems to prove existence and uniqueness of mild solutions. Proposition \ref{prop:equiv_mild_weak_sol} states that this mild solution is also a weak solution which proves in particular the existence of weak solutions. Additionally, we can even infer on the uniqueness of weak solutions if we can ensure that \eqref{eq:uniqueness_pde} admits no weak solution other than $0$. For the proof of the Proposition we need a technical lemma.
\begin{lem} \label{Lemma:equiv_mild_weak_sol_2}
Let us suppose Assumptions \ref{A1} and \ref{A4}. Fix an arbitrary function $\tilde{\Lambda}:[0,T]\times \mathbb{R}^d \rightarrow \mathbb{R}$ in $L^1([0,T] \times \mathbb{R}^d)$ and an arbitrary function $\tilde{b}:[0,T]\times \mathbb{R}^d \rightarrow \mathbb{R}^d$ in $L^1([0,T] \times \mathbb{R}^d)$.
 Define the function $v:[0,T]\times \mathbb{R}^d \rightarrow \mathbb{R}$ by 
\begin{align}
v(t,x) = &\int\limits_{\mathbb{R}^d}  p(0,x_0,t,x) {\bf u_0}(\dif x_0) + \int\limits_0^t \int\limits_{\mathbb{R}^d}\tilde{\Lambda}(s,x_0) p(s,x_0,t,x) \dif x_0 \dif s \notag \\
&-\sum\limits_{j=1}^d \int\limits_0^t \int\limits_{\mathbb{R}^d} \tilde{b}_j(s,x_0)\partial_{x_0,j} p(s,x_0,t,x) \dif x_0 \dif s.
\label{eq:mild_sol_def_b_and_lambda_hat}
\end{align}
Then  for any $\varphi \in C_0^{\infty}(\mathbb{R}^d)$ and for any $t\in [0,T]$
\begin{align}
\int\limits_{\mathbb{R}^d} \varphi(x)v(t,x)\dif x = &\int\limits_{\mathbb{R}^d} \varphi(x) {\bf u_0}(\dif x) + \int\limits_0^t \int\limits_{\mathbb{R}^d} v(s,x)L_s \varphi(x) \dif x \dif s   \notag \\
& +\sum\limits_{j=1}^d \int\limits_0^t \int\limits_{\mathbb{R}^d} \partial_j (\varphi (x)) \tilde{b}_j(s,x) \dif x \dif s  + \int\limits_0^t \int\limits_{\mathbb{R}^d} \varphi(x)  \tilde{\Lambda}(s,x)\dif x\dif s. \label{eq:weak_sol_def_b_and_lambda_hat}
\end{align}
\end{lem}
\begin{proof}[\nopunct Proof.] The idea of the proof is to generalize the statement formulated in Step 0., below,  by approximating $b$
and $\Lambda$ with smooth functions. It is thus natural to structure the proof as follows. \\ \\
\textbf{Step 0. Fubini and the fundamental solution}\\ \\	
\noindent Define the function $\tilde{v}:[0,T]\times \mathbb{R}^d \rightarrow \mathbb{R}$ by
\begin{equation} \label{eq:varation_mild_sol_Lemma}
\tilde{v}(s,x)=\int\limits_{\mathbb{R}^d}  p(0,x_0,s,x) {\bf u_0}(\dif x_0)+ \int\limits_0^s \int\limits_{\mathbb{R}^d} \tilde{\Lambda}(r,x_0) p(r,x_0,s,x) \dif x_0 \dif r. 
\end{equation}
Using Fubini's theorem and the definition of the fundamental solution, it is easy to check that $\tilde{v}$ satisfies for any test function $\varphi \in C_0^{\infty}(\mathbb{R}^d)$ and any $t\in [0,T]$
\begin{equation}\label{eq:weak_sol_rearranged}
\int\limits_0^t \int\limits_{\mathbb{R}^d} \tilde{v}(s,x)L_s \varphi(x) \dif x \dif s= 
\int\limits_{\mathbb{R}^d} \varphi(x)\tilde{v}(t,x)\dif x -\int\limits_{\mathbb{R}^d} \varphi(x)\tilde{v}(0,x)\dif x \notag \\ 
- \int\limits_0^t \int\limits_{\mathbb{R}^d} \varphi(x)  \tilde{\Lambda} (s,x) \dif x\dif s. 
\end{equation} 
\textbf{Step 1. Suppose that the mapping} 
$\mathbf{x \mapsto \tilde{\Lambda}(t,x)}$ \textbf{is bounded and that the map }\\
$\mathbf{(t,x) \mapsto \tilde{b}_j(t,x)}$ 
\textbf{belongs to} $\mathbf{C_b^{0,1}}$ \textbf{ for all } $\mathbf{j=1,\ldots,d}$.\\ \\   
In this case, \eqref{eq:mild_sol_def_b_and_lambda_hat} can be manipulated by integration by parts to read
$$
v(s,x) 
= \int\limits_{\mathbb{R}^d} p(0,x_0,s,x) \mathbf{u}_0(\dif x_0) \\
 + \int\limits_0^s \int\limits_{\mathbb{R}^d} \left[\tilde{\Lambda}(r,x_0)+ \left(\sum\limits_{j=1}^d \partial_{x_0,j} \tilde{b}_j(r,x_0))\right) \right] p(r,x_0,s,x) \dif x_0 \dif r.
$$
Note that the boundary term in the integration by parts vanishes because $\tilde{b}$ is bounded and $p$ satisfies inequality \eqref{eq:1011}. Because $\tilde{b}$ has bounded derivatives, the function, $\Psi:[0,T]\times \mathbb{R}^d \rightarrow \mathbb{R}$, such that $\Psi(t,x):=\tilde{\Lambda}(t,x)+ \sum_{j=1}^d \partial_{j}\tilde{b}_j(t,x)$ is bounded, too. Thus we can use the statement of Step
 0. with $\tilde{\Lambda}(t,x):=\Psi(t,x)$.
Thus we immediately obtain the desired equation \eqref{eq:weak_sol_def_b_and_lambda_hat} by inserting the definition of $\Psi$ and using the linearity of the integral and integrating by parts in the reverse sense of above computation. \\ \\ 
\noindent \textbf{Step 2. Assume that the maps}  $\mathbf{(t,x) \mapsto \tilde{b}_j(t,x)}$ \textbf{ and } $\mathbf{(t,x) \mapsto \tilde{\Lambda}(t,x)}$ \textbf{are in} $\mathbf{L^1([0,T]\times \R^d)}$ \textbf{ for all } $\mathbf{j=1,\ldots, d}$.\\ \\
Because $C^\infty_0([0,T]\times \R^d)$ is dense in $L^1([0,T]\times \R^d)$, there  exist sequences $(b_{n,j})_{n\in \N}$, $j=1,\ldots, d$ and $(\Lambda_n)_{n\in \N}$, in $C^\infty_0([0,T]\times \R^d)$, such that $(b_{n,j})_{n\in \N}$ converges to $\tilde{b}_{j}$ and $(\Lambda_n)_{n\in \N}$ converges to $\tilde{\Lambda}$ with respect to $\left\Vert \cdot \right\Vert_{L^1([0,T]\times \R^d)}$. For any $n \in \mathbb{N}$, we define the function $v_n:[0,T]\times \mathbb{R}^d \rightarrow \mathbb{R}$ by
$$
 v_n(s,x) = \int\limits_{\mathbb{R}^d} p(0,x_0,s,x) \mathbf{u}_0(\dif x_0) +
 G_n(s,x) + F_{n,j}(s,x),$$ 
where
\begin{eqnarray*}
G_{n} (s,x)&:=&  \int\limits_0^s \int\limits_{\mathbb{R}^d} \Lambda_n (r,x_0) p(r,x_0,s,x) \dif x_0 \dif r\\ 
 F_{n,j}(s,x)&:=&\int\limits_0^s \int\limits_{\mathbb{R}^d} b_{n,j} (r,x_0)\partial_{x_0,j} p(r,x_0,s,x) \dif x_0 \dif r.
\end{eqnarray*}
We will prove that $F_{n,j}$ converges, as $n$ tends to infinity, in $\left\Vert \cdot \right\Vert_{L^1([0,T]\times \R^d)}$, to $F_{j}: [0,T]\times \R^d \rightarrow \mathbb{R}$ such that 
\begin{align*}
F_j(s,x):=\int\limits_0^s \int\limits_{\mathbb{R}^d} \tilde{b}_j (r,x_0)\partial_{x_0,j} p(r,x_0,s,x) \dif x_0 \dif r\ .
\end{align*}
Indeed, by \eqref{eq:1011} and the fact that
$ q(r,x_0,s, x)$  is a probability density for almost all $r, x_0, s$ and Tonelli
we get
\begin{align*} 
 \Vert F_{n,j}(s,x) - F_j(s,x) \Vert_{L^1([0,T]\times \R^d)} &
 \leq \int\limits_0^T \int\limits_{\R^d} 
\int\limits_0^s \int\limits_{\mathbb{R}^d} \left\vert b_{n,j} (r,x_0)- \tilde{b}_j (r,x_0)\right\vert  
\left\vert \partial_{x_0,j} p(r,x_0,s,x)\right\vert
 \dif x_0 \dif r \dif x \dif s\\
&\leq \int\limits_0^T \int\limits_{\R^d} 
\int\limits_0^s \int\limits_{\mathbb{R}^d} \left\vert b_{n,j} (r,x_0)- \tilde{b}_j (r,x_0)\right\vert  \frac{C_u}{\sqrt{s-r}}q(r,x_0,s,x) \dif x_0 \dif r \dif x \dif s\\
&= \int\limits_0^T \int\limits_{\R^d} 
\int\limits_0^s \left\vert b_{n,j} (r,x_0)- \tilde{b}_j (r,x_0)\right\vert  \frac{C_u}{\sqrt{s-r}} 
\int\limits_{\mathbb{R}^d}q(r,x_0,s,x) \dif x
\dif r \dif x_0 \dif s \\
&=  C_u \int\limits_0^T \left\Vert b_{n,j} (r,\cdot)- \tilde{b}_j (r,\cdot)\right\Vert_{L^1}
\int\limits_r^T \frac{1}{\sqrt{s-r}} \dif s \dif r \\
&\leq 2C_u \sqrt{T}  \int\limits_0^T \left\Vert b_{n,j} (r,\cdot)- 
\tilde{b}_j (r,\cdot)\right\Vert_{L^1} \dif r = 
2C_u \sqrt{T} \left\Vert b_{n,j}-\tilde{b}_j \right\Vert_{L^1([0,T]\times \R^d)}. 
\end{align*}
To conclude, we now simply take the limit $n\rightarrow \infty$. With similar 
arguments, we prove that $G_{n}[0,T]\times \R^d \rightarrow \mathbb{R}$ converges, as $n$ tends to infinity, in $\left\Vert \cdot \right\Vert_{L^1([0,T]\times \R^d)}$, to $G: [0,T]\times \R^d \rightarrow \mathbb{R}$, where 
$$
G(s,x) :=\int\limits_0^s \int\limits_{\mathbb{R}^d} \tilde{\Lambda} (r,x_0) p(r,x_0,s,x) \dif x_0 \dif r.$$
From this, we conclude that $v_n$ converges in $\left\Vert \cdot \right\Vert_{L^1([0,T]\times \R^d)}$ to $v$ as defined in \eqref{eq:mild_sol_def_b_and_lambda_hat}. 
According to Step 1. of the proof, for any $n\in \mathbb{N}$ and any $\varphi \in C_0^{\infty}(\R^d)$
we also have 
\begin{align*}
\int\limits_{\mathbb{R}^d} \varphi(x)v_n(t,x)\dif x = &\int\limits_{\mathbb{R}^d} \varphi(x)v(0,x)\dif x + \int\limits_0^t \int\limits_{\mathbb{R}^d} v_n(s,x)L_s \varphi(x) \dif x \dif s   \notag \\
& +\sum\limits_{j=1}^d \int\limits_0^t \int\limits_{\mathbb{R}^d} \partial_j (\varphi (x)) b_{n,j}(s,x) \dif x \dif s  + \int\limits_0^t \int\limits_{\mathbb{R}^d} \varphi(x)  {\Lambda}_n(s,x)\dif x\dif s.
\end{align*}
We now take the limit as $n$ tends to infinity on both sides of this equation and only have to justify that we can exchange the limits and the integrals. For this, it suffices to observe that whenever a sequence of functions converges in any $L^1$-space, then multiplying this function with a bounded function preserves the $L^1$-convergence. 
\end{proof}
\noindent Now that we have completed the proof of Lemma \ref{Lemma:equiv_mild_weak_sol_2}, we turn to the proof of Proposition \ref{prop:equiv_mild_weak_sol}. 
\begin{proof}[Proof \nopunct] (of Proposition \ref{prop:equiv_mild_weak_sol}).
\begin{description}
\item{a)}
First suppose that $u$ is a mild solution  of PDE \eqref{eq:PDE_original} in the sense of Definition \ref{def_mild_sol}. \\
Then we set $\tilde{\Lambda}(s,x)=\Lambda(s,x,u(s,x))u(s,x)$ and 
$\tilde{b}(s,x)=b(s,x,u(s,x))u(s,x)$ which are both in $L^1([0,T]\times \R^d)$ because $\Lambda(s,x,u(s,x))$ and $b(s,x,u(s,x))$ are bounded by assumption and $u$ is in $L^1([0,T]\times \R^d)$ according to the definition of the mild solution. Hence we can apply Lemma \ref{Lemma:equiv_mild_weak_sol_2} 
which directly yields that $u$ is a  weak solution. 
\item{b)}
Now suppose that $u$ is a weak solution of PDE \eqref{eq:PDE_original} in the sense of Definition \ref{def_weak_sol}. \\  
Then, we define
\begin{align*}
v(t,x):=&\int\limits_{\mathbb{R}^d} p(0,x_0,t,x)\mathbf{u_0} ( \dif x_0)+ \int\limits_{0}^t \int\limits_{\mathbb{R}^d} u(s,x_0)\Lambda(s,x_0,u(s,x_0))p(s,x_0,t,x) \dif x_0 \dif s \\ 
&+ \sum\limits_{j=1}^d \int\limits_{0}^t \int\limits_{\mathbb{R}^d} u(s,x_0) (b_j(u,x_0,u(s,x_0)) \partial_{x_0,j}  p(s,x_0,t,x) \dif x_0 \dif s.
\end{align*}
The strategy of this proof is to establish that  $v=u$. Again, we apply Lemma \ref{Lemma:equiv_mild_weak_sol_2} with $\tilde{\Lambda}(s,x)=\Lambda(s,x,u(s,x))u(s,x)$ and $\tilde{b}(s,x)=b(s,x,u(s,x))u(s,x)$ on $v$ and infer that $v$ is a weak solution of the PDE
\begin{align}
\begin{cases}
\partial_t v(t,x) &= L_t^* v(t,x) -\sum\limits_{j=1}^d \partial_j (b_j(t,x,u(t,x))u(t,x)) + u(t,x)\Lambda(t,x,u(t,x)) \\
v(0,\cdot)&=u_0.
\end{cases} \label{eq:PDE_proof_equivalence_mild_weak}
\end{align}
By assumption, $u$ is a weak solution of PDE \eqref{eq:PDE_original}, so it is in particular a weak solution of \eqref{eq:PDE_proof_equivalence_mild_weak}. 
Hence $u-v$ is a solution of the PDE \eqref{eq:uniqueness_pde}.
 But as by Assumption~\ref{C} the unique weak solution of \eqref{eq:uniqueness_pde} is the constant function equal to zero, we conclude 
that $v=u$ is a mild solution of~\eqref{eq:PDE_original}.
\end{description}
\end{proof}

\section{Existence and uniqueness of a mild solution to the nonlinear PDE}\label{S4}
In this section Assumption~\ref{A4} is in force.
We fix some 
bounded Borel measurable functions
$\hat{\Lambda}:[0,T]\times \mathbb{R}^d  \rightarrow \mathbb{R}$ and $\hat{b}:[0,T]\times \mathbb{R}^d  \rightarrow \mathbb{R}^d$. 

We begin by introducing the notion of a measure-mild solution of the linearized equation
\begin{align}
\label{eq:linearized_PDE}
\begin{cases} 
\partial_t v(t,x) &= L_t^* v(t,x) -\sum \limits_{j=1}^d \partial_j (\hat{b}_j(t,x)v(t,x))+ \hat{\Lambda}(t,x)v(t,x), \\	
v(0,\cdot)&=\mathbf{u_0}.
\end{cases}
\end{align}
In fact \eqref{eq:linearized_PDE}  constitutes the
linearized version of PDE \eqref{eq:PDE_original}.

\subsection{Notion of measure-mild solution}
\label{SMeasureMild}

Recall that $p$ denotes the fundamental solution to \eqref{eq:FokkerPlanck}
as introduced in Definition \ref{D3}.
\begin{defi} \label{def:measure_mild_solution}
Assume \ref{A4} and that $\hat{b}$ and $\hat{\Lambda}$ are bounded. A measure-valued map $\mu : [0,T] \rightarrow \mathcal{M}_f(\mathbb{R}^d)$ will be called \textbf{measure-mild solution} of \eqref{eq:linearized_PDE} if  for all $t\in [0,T]$
\begin{align}
\label{eq:measure_mild_sol_def}
 \mu_t (\dif x) =  &\int\limits_{\mathbb{R}^d} p(0,x_0,t,x)\dif x  {\bf u_0} ( \dif x_0) + \int\limits_0^t \int\limits_{\mathbb{R}^d}\hat{\Lambda}(r,x_0) p(r,x_0,t,x) \dif x \mu_r(\dif x_0) \dif r \notag \\
&+ \sum\limits_{j=1}^d \int\limits_0^t \int\limits_{\mathbb{R}^d} \hat{b}_j(r,x_0) \partial_{x_0,j} p(r,x_0,t,x)\dif x \mu_r(\dif x_0) \dif r.
\end{align}
\end{defi} 
\begin{rem} \label{rem:measure_mild_def_test} \hfill 
\begin{enumerate}
\item Note that a measure is characterized by the integrals of all smooth functions with bounded support. This is why the definition of a measure-mild solution in  \eqref{eq:measure_mild_sol_def} is equivalent to require that for any $\varphi \in C_0^{\infty}$ and for any $t\in [0,T]$
\begin{align*}
\int\limits_{\mathbb{R}^d} \varphi(x) \mu_t(\dif x) =  &\int\limits_{\mathbb{R}^d} \Big ( \int\limits_{\mathbb{R}^d} \varphi(x)  p(0,x_0,t,x)\dif x \Big ) {\bf u_0} ( \dif x_0) \notag \\
&+ \int\limits_0^t \int\limits_{\mathbb{R}^d}\hat{\Lambda}(r,x_0) \Big ( \int\limits_{\mathbb{R}^d} \varphi(x) p(r,x_0,t,x)\dif x \Big ) \mu_r(\dif x_0) \dif r \notag \\
&+ \sum\limits_{j=1}^d \int\limits_0^t \int\limits_{\mathbb{R}^d} \hat{b}_j(r,x_0) \Big ( \int\limits_{\mathbb{R}^d} \varphi(x) \partial_{x_0,j} p(r,x_0,t,x)\dif x \Big ) \mu_r(\dif x_0) \dif r.
\end{align*}
\item If one compares the definition of measure-mild solutions, setting
$$\hat{b}(s,x) = b(s,x,u(s,x)), \ 
\hat{\Lambda}(s,x) = \Lambda(s,x,u(s,x)), $$
to the one of mild solutions in Definition \ref{def_mild_sol}, it first appears that the only difference being that $u(t,x)\dif x$ has been replaced by $\mu(t,\dif x)$.  
 In particular, if $u$ is a mild solution of \eqref{eq:PDE_original},
see Definition \ref{def_mild_sol},
then $\mu_t(\dif x) := u(t,x) \dif x$  is a measure-mild solution
of \eqref{eq:linearized_PDE}. 
\end{enumerate}
\end{rem}
In Proposition \ref{prop:uniqueness_measure_mild} we will  show
under Assumption \ref{A4} uniqueness of
measure-mild solution of the {\it linear} equation
\eqref{eq:linearized_PDE}.

\subsection{Proof of Theorem \ref{thm:e_and_u}}

We will adapt the construction of a mild solution given in \cite{LOR3} to the case of PDE \eqref{eq:PDE_original}.
We will apply the Banach fixed-point theorem to a suitable mapping in order to prove existence and uniqueness of a mild solution on a \textit{small} time interval. \\ 
We will frequently make use of the notation $L^1([t_1,t_2],L^1(\mathbb{R}^d))$ for $t_1<t_2$, which is the set of functions $f$ on $[t_1,t_2]$ with values in $L^1(\mathbb{R}^d)$ such that $\Vert f \Vert_1<\infty$, where this norm is defined as 
\begin{align}
\Vert f \Vert_1:=\int\limits_{t_1}^{t_2} \Vert f(s)(\cdot) \Vert_{L^1(\mathbb{R}^d)} \dif s \label{eq:def_norm_l1}.
\end{align}
Fix $\phi \in L^1(\mathbb{R}^d)\cap L^{\infty}(\mathbb{R}^d), r \in [0,T[$ and $\tau \in ]0,T-r].$
We define $\widehat{u_0}(r,\phi):[r,r+\tau]\times \mathbb{R}^d \rightarrow \R $
 by
\begin{align}
\widehat{u_0}(r,\phi)(t,x)&:=\int\limits_{\mathbb{R}^d} p(r,x_0,t,x) \phi(x_0)\dif x_0. \label{eq:def_u_zero_hat}
\end{align}
We will often forget to indicate the $r$ parameter and note $\widehat{u_0}(\phi)$ instead of $\widehat{u_0}(r,\phi)$. 
By \eqref{eq:1011a}, the following bounds hold 
\begin{equation} \label{E30} 
  \Vert \widehat{u_0}(r,\phi)  \Vert_1 \le \Vert \phi \Vert_1\quad \textrm{and}\quad 
\Vert \widehat{u_0}(r,\phi) \Vert_\infty \le C_u \Vert \phi \Vert_\infty.
\end{equation}
We define $\Pi: L^1([r,r+\tau],L^1(\mathbb{R}^d)) \rightarrow  L^1([r,r+\tau],L^1(\mathbb{R}^d))$ by
\begin{align}
\Pi(r,v)(t,x):=&\int\limits_r^t \int\limits_{\mathbb{R}^d} p(s,x_0,t,x) \hat{\Lambda}(s,x_0,v+\widehat{u_0}(\phi)) \dif x_0  \dif s \notag \\
&+\sum\limits_{j=1}^d  \int\limits_r^t \int\limits_{\mathbb{R}^d} \partial_{x_0,j} p(s,x_0,t,x) \hat{b}_j(s,x_0,v+\widehat{u_0}(\phi)) \dif x_0  \dif s, 
\label{eq:def_pi}
\end{align}
where we have used the shorthand notation
\begin{align*}
\hat{\Lambda}(s,x_0,v+\widehat{u_0}(\phi)) 
&:= \Lambda(s,x_0,v(s,x_0)+\widehat{u_0}(\phi)(s,x_0)) (v(s,x_0)+\widehat{u_0}(\phi)(s,x_0)) ,
\notag \\
\hat{b}(s,x_0,v+\widehat{u_0}(\phi)) 
&:= b(s,x_0,v(s,x_0)+\widehat{u_0}(\phi)(s,x_0)) (v(s,x_0)+\widehat{u_0}(\phi)(s,x_0)).
\end{align*}
We remark that the map $\Pi$ depends on $r, \tau$ but this
will be omitted in the sequel. 
We will denote the centered closed ball of radius $M$ in $L^1(\mathbb{R}^d)$ by $B[0,M]$  and the closed centered ball in $L^{\infty}([r,r+\tau]\times \mathbb{R}^d, \mathbb{R})$ by $B_{\infty}[0,M]$. The following lemma will illuminate why we have engaged in defining these objects.
\begin{lem} \label{lem:fixed_point_prep}
Assume \ref{A1}, \ref{A4}, \ref{A5}, \ref{A6} and \ref{A7}.
\begin{enumerate}
\item For every $(r, \tau)$, such that $r \in [0,T[, \tau \in ]0,T-r]$,
 $L^1([r,r+\tau],B[0,M])\cap B_{\infty}[0,M]$ is a 
closed subset of $L^1([r,r+\tau]\times \R^d) $.  
In particular, it is a complete metric space equipped with $\left\Vert \cdot \right\Vert_1$, defined in \eqref{eq:def_norm_l1}, with $t_1 = r$ and 
$t_2 = r+\tau$.
\item There exists $\tau >0$ such that for any $r\in [0,T-\tau]$, 
we have 
$$\Pi (L^1([r,r+\tau],B[0,M])\cap B_{\infty}[0,M])\subset (L^1([r,r+\tau],B[0,M])\cap B_{\infty}[0,M]).$$
\item Let $\tau$ be as in previous item. There exists an integer $k_0 \in \mathbb{N}$, such that $\Pi
^{k_0}$ is a contraction on $L^1([r,r+\tau],B[0,M])
\cap B_{\infty}[0,M]$ for any 
$r\in [0,T-\tau]$.
\end{enumerate}
In particular, for any $r\in [0,T-\tau]$, there exists a unique fixed-point of $\Pi$. This fixed-point is in 
$L^1([r,r+ \tau],B[0,M])\cap B_{\infty}[0,M].$
\end{lem}
\begin{proof}[Proof \nopunct.]
\begin{enumerate}
\item As for the first claim, we note that $L^1([r,r+\tau],B[0,M])$ is a complete space with respect to the norm defined in \eqref{eq:def_norm_l1}. Hence it suffices to show that $B_{\infty}[0,M]$ is closed with respect to this norm. So let $(f_n)_{n \in \mathbb{N}}\subset B_{\infty}[0,M]$ be such that $f_n \rightarrow f$ with respect to \eqref{eq:def_norm_l1}. Then we can extract a subsequence $(f_{n_k})_{k \in \mathbb{N}}$ which converges almost everywhere pointwise 
on $[0,T] \times \R^d$,
against $f$. Because $(f_n)$ are uniformly bounded by $M$, so is $f$. Now we only need to recall that the intersection of closed sets (in this case $L^1([r,r+\tau], B[0,M])$ and $B_{\infty}[0,M]$) is closed.\\ 
\item Let $(r, \tau)$, be fixed for the moment.
 Without loss of generality we assume 
$M= 1$.
 Let $v \in L^1([r,r+\tau],B[0,1])\cap B_{\infty}[0,1]$.
 Now, the triangle inequality yields
\begin{align*}
\Vert\Pi(v)\Vert_{1} =  &\int\limits_r^{r+\tau}  \Vert\Pi(v)(t,\cdot)\Vert_{L^1(\mathbb{R}^d)}  \dif t \\
\leq &\int\limits_r^{r+\tau}  \Vert\int\limits_r^t \int\limits_{\mathbb{R}^d} p(s,x_0,t,\cdot) \hat{\Lambda}(s,x_0,v+\widehat{u_0}(\phi)) \dif x_0  \dif s \Vert_{L^1(\mathbb{R}^d)} \dif t \\
&+ \sum\limits_{j=1}^d \int\limits_r^{r+\tau}  \Vert \int\limits_r^t \int\limits_{\mathbb{R}^d} \partial_{x_0,j} p(s,x_0,t,\cdot) \hat{b}_j(s,x_0,v+\widehat{u_0}(\phi)) \dif x_0  \dif s \Vert_{L^1(\mathbb{R}^d)}  \dif t.
\end{align*}
Note that for the first term on the right-hand side of this inequality, we already have an upper bound from equation~(3.28) in the proof of Lemma~3.7  in~\cite{LOR3} which reads
\begin{align} \label{E31}
\int\limits_r^{r+\tau}  \Vert\int\limits_r^t \int\limits_{\mathbb{R}^d} p(s,x_0,t,\cdot) \hat{\Lambda}(s,x_0,v+\widehat{u_0}(\phi)) \dif x_0  \dif s \Vert_{L^1(\mathbb{R}^d)} \dif t \leq 2 M_{\Lambda}\tau^2.
\end{align}
As for the second term, we restrict our discussion to the first summand, $j=1$,  because all summands can be estimated in the same way. We adapt the estimation of the $W^{1,1}$-norm in the proof of \cite[Lemma 3.7]{LOR3} :
\begin{align*}
&\int\limits_{\mathbb{R}^d} 
\big \vert \int\limits_r^t \int\limits_{\mathbb{R}^d} \partial_{x_0,1} p(s,x_0,t,x) \hat{b}_1(s,x_0,v+\widehat{u_0}(\phi)) \dif x_0  \dif s \Big\vert \dif x \\
&\leq M_{b} \int\limits_{\mathbb{R}^d}  \int\limits_r^t \int\limits_{\mathbb{R}^d} \left\vert \partial_{x_0,1} p(s,x_0,t,x)\right\vert \left\vert (v+\widehat{u_0}(\phi))(s,x_0) \right\vert  \dif x_0  \dif s \dif x \\
&\stackrel{\text{Tonelli}}{=} M_{b} \int\limits_r^t \int\limits_{\mathbb{R}^d}   \int\limits_{\mathbb{R}^d} \left\vert \partial_{x_0,1} p(s,x_0,t,x)\right\vert \left\vert (v+\widehat{u_0}(\phi))(s,x_0) \right\vert  \dif x  \dif x_0 \dif s\\
&\stackrel{\text{ineq.}\eqref{eq:1011}}{\leq } M_{b} C_u \int\limits_r^t \frac{1}{\sqrt{t-s}} \int\limits_{\mathbb{R}^d} \left\vert (v+\widehat{u_0}(\phi))(s,x_0) \right\vert   \int\limits_{\mathbb{R}^d} q(s,x_0,t,x) \dif x \dif x_0 \dif s\\ 
&= M_{b}C_u \int\limits_r^t \frac{1}{\sqrt{t-s}} \int\limits_{\mathbb{R}^d} \left\vert (v+\widehat{u_0}(\phi))(s,x_0) \right\vert  \dif x_0 \dif s\\
&= M_{b}C_u  \int\limits_r^t \frac{1}{\sqrt{t-s}} \underbrace{\left\Vert (v+\widehat{u_0}(\phi))(s,\cdot) \right\Vert_{L^1(\mathbb{R}^d)}}_{\leq \left\Vert v(s,\cdot) \right\Vert_{L^1(\mathbb{R}^d)}+ \left\Vert \widehat{u_0}(\phi)(s,\cdot) \right\Vert_{L^1(\mathbb{R}^d)} \leq 1+ \left\Vert \phi \right\Vert_1} \dif s \\
&\leq (1+ \left\Vert \phi \right\Vert_1) M_b C_u \int\limits_r^t \frac{1}{\sqrt{t-s}} \dif s 
\leq  2(1+ \left\Vert \phi \right\Vert_1) M_b C_u  \sqrt{\tau}.
\end{align*}
This and \eqref{E31} give 
\begin{align}
\Vert\Pi(v)\Vert_{1}\leq 2M_\Lambda \tau^2+ 4 d M_b C_u  \sqrt{\tau}= 
2\sqrt{\tau} (M_{\Lambda} \tau^{\frac{3}{2}} +2 d M_b C_u).
\end{align}
For the $L^{\infty}$-norm of $\Pi(v)$, we repeat the same calculation,
which
yields that there also is a constant $\bar{C}:=\bar{C}(C_u,c_u, \Lambda, b, d)>0$ such that  
\begin{align}
\left\Vert \Pi(v)\right\Vert_{\infty}\leq \bar C \sqrt{\tau}.
\end{align}
In particular, $\Vert\Pi(v)\Vert_{1}$ and $\Vert\Pi(v)\Vert_{\infty}$ converge to zero when $\tau$ converges to zero. So we can choose $\tau$ small enough to ensure that $\Vert\Pi(v)\Vert_{1}\leq 1$ and $\Vert\Pi(v)\Vert_{\infty} \leq 1$. This completes the proof of claim 2.
\item
To verify  
the contraction property on $L^1([r,r+\tau],B[0,M])\cap B_{\infty}[0,M]$ for any $r\in [0,T-\tau]$, we fix $v_1,v_2 \in L^1([r,r+\tau],B[0,M]))\cap B_{\infty}[0,M]$ and $t \in [0,T-\tau]$. For readability, we use the shorthand notation 
$  \Lambda^i(s,x_0):=   \Lambda(s,x_0,(\widehat{u_0}(r,\phi)+v_i)(s,x_0))$ and similarly 
$b^i_j(s,x_0):=\hat{b}_j(s,x_0,(\widehat{u_0}(r,\phi)+v_i)(s,x_0))$ for $i=1,2$
and $ 1 \le j \le d$.
  Then, for $t \in [r,r+\tau]$, we compute
\begin{align*}
& \Vert \Pi(v_1)(t,\cdot)-\Pi(v_2)(t,\cdot) \Vert_{L^1(\mathbb{R}^d)}   \\ 
\le &\int\limits_{\mathbb{R}^d} \Big\vert \int\limits_r^t \int\limits_{\mathbb{R}^d}p(s,x_0,t,x)(\Lambda^1(s,x_0)(\widehat{u_0}+v_1)(s,x_0)-\Lambda^2
(s,x_0)(\widehat{u_0}+v_2)(s,x_0)) \dif x_0 \dif s \Big\vert \dif x \\
&+ \sum\limits_{j=1}^d \int\limits_{\mathbb{R}^d} \Big\vert \int\limits_r^t \int\limits_{\mathbb{R}^d} \partial_{x_0,j}p(s,x_0,t,x) (b_j^1(s,x_0)
(\widehat{u_0}+v_1)(s,x_0)-b_j^2(s,x_0)(\widehat{u_0}+v_2)(s,x_0))
 \dif x_0  \dif s \Big\vert \dif x \\
&=: I_1 + I_2.
\end{align*}
For the rest of the proof, we will forget the argument $(s,x_0)$ for simplicity.
Concerning $I_1$, we compute as follows:
\begin{align*}
I_1&=\int\limits_{\mathbb{R}^d} \Big\vert \int\limits_r^t \int\limits_{\mathbb{R}^d}p(s,x_0,t,x)(\Lambda^1(\widehat{u_0}+v_1)-\Lambda^2(\widehat{u_0}+v_2)) \dif x_0 \dif s \Big\vert \dif x \\
&\leq \int\limits_{\mathbb{R}^d} \int\limits_r^t \int\limits_{\mathbb{R}^d}p(s,x_0,t,x)\Big\vert(\Lambda^1(\widehat{u_0}+v_1)-\Lambda^2(\widehat{u_0}+v_2)\Big\vert  \dif x_0 \dif s \dif x \\
&= \int\limits_r^t \int\limits_{\mathbb{R}^d}\left\vert(\Lambda^1(\widehat{u_0}+v_1)-\Lambda^2(\widehat{u_0}+v_2)) +\Lambda^1v_2-\Lambda^1 v_2 \right\vert \dif x_0 \dif s \\
&=\int\limits_r^t \int\limits_{\mathbb{R}^d}\left\vert(\Lambda^1-\Lambda^2)\widehat{u_0} +\Lambda^1 (v_1-v_2)+(\Lambda^1-\Lambda^2) v_2 \right\vert \dif x_0 \dif s \\
&\leq \int\limits_r^t \int\limits_{\mathbb{R}^d} \underbrace{\left\vert\Lambda^1-\Lambda^2\right\vert}_{\leq L_{\Lambda}\left\vert v_1-v_2 \right\vert } \underbrace{\left\vert\widehat{u_0}\right\vert}_{\leq M} +\underbrace{\left\vert\Lambda^1\right\vert}_{\leq M_{\Lambda}} \left\vert v_1-v_2\right\vert+\underbrace{\left\vert\Lambda^1-\Lambda^2\right\vert}_{\leq L_{\Lambda}\left\vert v_1-v_2 \right\vert } \underbrace{\left\vert v_2 \right\vert}_{\leq M} \dif x_0 \dif s \\
&= (2L_{\Lambda}M+M_{\Lambda}) \int\limits_r^t \left\Vert v_1(s,\cdot )-v_2(s,\cdot ) \right\Vert_{L^1(\mathbb{R}^d)} \dif s.
\end{align*}
For $I_2$, we again restrict our discussion to the first summand. We can estimate
\begin{align*}
&\int\limits_{\mathbb{R}^d} \Big\vert \int\limits_r^t \int\limits_{\mathbb{R}^d} \partial_{x_0,1} p(s,x_0,t,x) (b_1^1(\widehat{u_0}+v_1)-b_1^2(\widehat{u_0}+v_2)) \dif x_0  \dif s \Big\vert \dif x \\
& 
\leq \int\limits_r^t \int\limits_{\mathbb{R}^d}  \int\limits_{\mathbb{R}^d}\big\vert b_1^1(\widehat{u_0}+v_1)-b_1^2(\widehat{u_0}+v_2) \big\vert \big\vert  \partial_{x_0,1} p(s,x_0,t,x)\big\vert   \dif x  \dif x_0  \dif s \\
&\stackrel{\eqref{eq:1011}}{\leq} \int\limits_r^t \frac{1}{\sqrt{t-s}} \int\limits_{\mathbb{R}^d} \left\vert b_1^1(\widehat{u_0}+v_1)-b_1^2(\widehat{u_0}+v_2) 
\right\vert  \int\limits_{\mathbb{R}^d} 
C_u q(s,x_0,t,x)
  \dif x  \dif x_0  \dif s \\
&\leq C_u \int\limits_r^t \frac{1}{\sqrt{t-s}} \int\limits_{\mathbb{R}^d} \left\vert b_1^1(\widehat{u_0}+v_1)-b_1^2(\widehat{u_0}+v_2) \right\vert \dif x_0  \dif s.
\end{align*}
By similar estimates as for $I_1$, previous expression can be shown
to be bounded by
$$ C_u (2L_{b}M+M_{b}) \int\limits_r^t \frac{1}{\sqrt{t-s}} \Vert v_1(s,\cdot )-v_2(s,\cdot ) \Vert_{L^1(\mathbb{R}^d)} \dif s.$$
Until now, we have established that there exists a constant $C>0$ 
which can change from line to line such that
\begin{eqnarray}
 \Vert \Pi(v_1)(t,\cdot)-\Pi(v_2)(t,\cdot) \Vert_{L^1(\mathbb{R}^d)}    
&\leq& C \int\limits_r^t \left(1+ \frac{1}{\sqrt{t-s}}\right) \Vert v_1(s,\cdot )-v_2(s,\cdot ) \Vert_{L^1(\mathbb{R}^d)} \dif s \nonumber  \\
 &\leq& C \int\limits_r^t \frac{1}{\sqrt{t-s}} \Vert v_1(s,\cdot )-v_2(s,\cdot ) \Vert_{L^1(\mathbb{R}^d)} \dif s.  
\label{RFTrick}
\end{eqnarray}
We now try to find a power of $\Pi$ which is a contraction. For this, we iterate $\Pi$, aiming to apply
 a J.B. Walsh iteration  trick, see e.g. the proof of Theorem
3.2 in \cite{walsh1986introduction}.
For $t \in [r,r+\tau]$ we have 
\begin{align} \label{RFTrickBis}
&\Vert \Pi^2(v_1)(t,\cdot)-\Pi^2(v_2)(t,\cdot) \Vert_{L^1(\mathbb{R}^d)}    \nonumber \\
 \leq & C^2 \int\limits_r^t \int\limits_r^s \frac{1}{\sqrt{t-s}}\frac{1}{\sqrt{s-\theta}} \Vert v_1(\theta,\cdot )-v_2(\theta,\cdot ) \Vert_{L^1(\mathbb{R}^d)} \dif \theta \dif  s \nonumber   \\
 = & C^2 \int\limits_r^t \int\limits_{\theta}^t  \frac{1}{\sqrt{t-s}}\frac{1}{\sqrt{s-\theta}} \Vert v_1(\theta,\cdot )-v_2(\theta,\cdot ) \Vert_{L^1(\mathbb{R}^d)} \dif  s \dif \theta \nonumber   \\
= & C^2 \int\limits_r^t \Big( \int\limits_{\theta}^t \frac{1}{\sqrt{t-s}}\frac{1}{\sqrt{s-\theta}} \dif  s \Big) \Vert v_1(\theta,\cdot )-v_2(\theta,\cdot ) \Vert_{L^1(\mathbb{R}^d)} \dif \theta \\
= & C^2 \int\limits_r^t \Big( \int\limits_{0}^{t-\theta}  \frac{1}{\sqrt{(t-\theta)-\omega}}\frac{1}{\sqrt{\omega}} \dif  \omega \Big)  \Vert v_1(\theta,\cdot )-v_2(\theta,\cdot ) \Vert_{L^1(\mathbb{R}^d)} \dif \theta  \nonumber  \\
= & \pi C^2 \int\limits_r^t \Vert v_1(\theta,\cdot )-v_2(\theta,\cdot ) \Vert_{L^1(\mathbb{R}^d)} \dif \theta.  \nonumber 
\end{align}
\noindent Indeed the last equality 
follows because
$$ \int\limits_{0}^{t-\theta}  \frac{1}{\sqrt{(t-\theta)-\omega}}\frac{1}{\sqrt{\omega}} \dif  \omega = B\left(\frac{1}{2}, \frac{1}{2}\right) = \pi,$$
where $(x,y) \mapsto B(x,y)$ is the Beta function.

\noindent Actually, this is the induction basis
 of a mathematical induction over $k\in \mathbb{N}$ to establish
\begin{align}
 \Vert \Pi^{2k}(v_1)(t,\cdot)-\Pi^{2k}(v_2)(t,\cdot) 
\Vert_{L^1(\mathbb{R}^d)}   \dif s
\leq  \pi^k C^{2k} \int\limits_r^t \frac{(t-s)^{k-1}}{(k-1)!} \Vert v_1(s,\cdot )-v_2(s,\cdot ) \Vert_{L^1(\mathbb{R}^d)} \dif s. \label{eq:power_of_pi_is_contraction}
\end{align}
So let us do the inductive step. The third line comes from \eqref{RFTrickBis}. 
\begin{align*}
& \Vert \Pi^{2k+2}(v_1)(t,\cdot)-\Pi^{2k+2}(v_2)(t,\cdot) 
\Vert_{L^1(\mathbb{R}^d)}    \\
= & \Vert \Pi^2 \Pi^{2k}(v_1)(t,\cdot)- \Pi^2 \Pi^{2k}(v_2)(t,\cdot) 
\Vert_{L^1(\mathbb{R}^d)}   \\
\leq & \pi C^2 \int\limits_{r}^{t} \Vert \Pi^{2k}(v_1)(\theta,\cdot)-\Pi^{2k}(v_2)(\theta,\cdot) \Vert_{L^1(\mathbb{R}^d)}     \dif \theta \\
\leq &\pi C^2 \int\limits_{r}^{t} \pi^{k} C^{2k} \int\limits_r^{\theta} \frac{(\theta-s)^{k-1}}{(k-1)!} \Vert v_1(s,\cdot )-v_2(s,\cdot ) \Vert_{L^1(\mathbb{R}^d)} \dif s \dif \theta \\
\leq &\pi^{k+1} C^{2(k+1)} \int\limits_r^t \int\limits_s^t \frac{(\theta-s)^{k-1}}{(k-1)!} \Vert v_1(s,\cdot )-v_2(s,\cdot ) \Vert_{L^1(\mathbb{R}^d)} \dif \theta \dif s \\
= & \pi^{k+1} C^{2(k+1)} \int\limits_r^t \frac{1}{(k-1)!} 
\int\limits_s^t (\theta-s)^{k-1}  \dif \theta 
\Vert v_1(s,\cdot )-v_2(s,\cdot ) \Vert_{L^1(\mathbb{R}^d)} \dif s \\
=& \frac{\pi^{k+1} C^{2(k+1)}}{(k-1)!}  \int\limits_r^t  \frac{(t-s)^k}{k}\Vert v_1(s,\cdot )-v_2(s,\cdot ) \Vert_{L^1(\mathbb{R}^d)} \dif s .
\end{align*}
Now that we have established equation \eqref{eq:power_of_pi_is_contraction}, we can use it to show that a power of $\Pi$ is a contraction: 
\begin{equation*} \label{E33bis}
\Vert \Pi^{2k}(v_1)(t,\cdot)-\Pi^{2k}(v_2)(t,\cdot) \Vert_{L^1(\mathbb{R}^d)}  
\leq  \frac{\pi^k T^{k-1} C^{2k}}{(k-1)!} \int\limits_r^t \Vert v_1(s,\cdot )-
v_2(s,\cdot ) \Vert_{L^1(\mathbb{R}^d)} \dif s. 
\end{equation*}
So, integrating from $r$ to $r +\tau$, we get
\begin{equation}
 \Vert \Pi^{2k}(v_1)-\Pi^{2k}(v_2) \Vert_{1}  \le 
 \frac{\pi^k T^{k} C^{2k}}{(k-1)!} \Vert v_1 - v_2 \Vert_{1}.
\end{equation}
Now we can conclude that there exists a $k_0 \in \mathbb{N}$ such that $\Pi^{2k_0}$ is a contraction because the exponential growth of $\pi^k T^k C^{2k}$
 is dominated by the factorial $(k-1)!$
Finally, $\Pi$ admits a unique fixed point as a contraction on a Banach space by Banach's fixed point theorem.
\end{enumerate}
\end{proof}
\noindent Now that we have constructed fixed points for short time intervals, we are concerned with ``gluing'' these fixed points together.
\begin{lem} \label{Lemma:glue} 
We fix some 
bounded Borel measurable functions
$\hat{\Lambda}:[0,T]\times \mathbb{R}^d  \rightarrow \mathbb{R}$ and $\hat{b}:[0,T]\times \mathbb{R}^d  \rightarrow \mathbb{R}^d$.
We suppose \ref{A4}.
Fix $\tau >0$ such that there is an $N \in \mathbb{N}$ s.t. $N \tau=T$. 
Then a measure valued map $\mu:[0,T]\mapsto \mathcal{M}_f(\mathbb{R}^d)$ satisfies 
\begin{align}
\begin{cases}
\mu_t(\dif x)&=\int\limits_{\mathbb{R}^d} p(k \tau , x_0,t, x) \mu_{k \tau}(\dif x_0)\dif x\\
& \quad +\int\limits_{k\tau}^t \int\limits_{\mathbb{R}^d} p(s,x_0,t,x) \hat{\Lambda}(s,x_0) \mu_s(\dif x_0 )  \dif s \dif x  \\
& \hfill \quad   +\sum\limits_{j=1}^d \int\limits_{k\tau}^t \int\limits_{\mathbb{R}^d} \partial_{x_0,j} p(s,x_0,t,x) \hat{b}_j(s,x_0) \mu_s(\dif x_0 )   \dif s \dif x,\\
\mu_0(\cdot)&={\bf u_0},
\end{cases}
\label{eq:measure_mild_sol_glue}
\end{align}
for all $t\in [k\tau,(k+1)\tau]$ and $k \in \{0,\ldots,N-1\}$ if and only if $\mu$ is a measure-mild solution of \eqref{eq:linearized_PDE}, in the sense of Definition 
\ref{def:measure_mild_solution}.
\end{lem}
\begin{proof}[Proof \nopunct.] 
\begin{enumerate}
\item Assume first that $\mu$ satisfies \eqref{eq:measure_mild_sol_glue} for all $t\in [k\tau,(k+1)\tau]$ and $k \in \{0,\ldots,N-1\}$.
 \\ \\
We will now show that this implies that $\mu$ is a measure-mild solution of \eqref{eq:linearized_PDE} in the sense of \eqref{eq:measure_mild_sol_def}. We will perform a mathematical induction to show that the statement
\begin{align*}
(S_n) \begin{cases}
\mu_t(\dif x)=&\int_{\mathbb{R}^d} p(0,x_0,t,x) {\bf u_0} (\dif x_0)\dif x +\int\limits_{0}^t \int\limits_{\mathbb{R}^d} p(s,x_0,t,x) \hat{\Lambda}(s,x_0) \mu_s(\dif x_0 )  \dif s \dif x \notag \\
&+ \sum\limits_{j=1}^d \int\limits_{0}^t \int\limits_{\mathbb{R}^d} \partial_{x_0,j} p(s,x_0,t, x) \hat{b}_j(s,x_0) \mu_s(\dif x_0 ) \dif s \dif x \qquad \text{ for all } t \in [0,n\tau]
\end{cases}
\end{align*} 
holds for any $n\in \{1,\ldots,N\}$ .
For the induction basis, it suffices to note that if we set $k=0$, equation \eqref{eq:measure_mild_sol_def} coincides with \eqref{eq:measure_mild_sol_glue}. Now suppose that ($S_{n-1}$) holds for some $n \in \{1,\ldots, N\}$. In particular, we have for $t=(n-1)\tau$
\begin{align}
\mu_{(n-1)\tau}(\dif x_0) =&\int_{\mathbb{R}^d} p(0,\widetilde{x_0},(n-1)\tau, x_0) 
{\bf u_0} (\dif \widetilde{x_0}) \dif x_0\notag \\
 &+\int\limits_{0}^{(n-1)\tau} \int\limits_{\mathbb{R}^d} p(s,\widetilde{x_0},(n-1)\tau,x_0) \hat{\Lambda}(s,\widetilde{x_0}) \mu_s(\dif \widetilde{x_0} )  \dif s \dif x_0 \notag \\
&+  \sum\limits_{j=1}^d \int\limits_{0}^{(n-1)\tau} \int\limits_{\mathbb{R}^d} \partial_{\widetilde{x_0},j} p(s,\widetilde{x_0},(n-1)\tau, x) \hat{b}_j(s,\widetilde{x_0}) \mu_s(\dif \widetilde{x_0} ) \dif s \dif x_0. \label{eq:measure_mild_glue_induction_hypothesis}
\end{align}
Recall that we also suppose the validity of equation \eqref{eq:measure_mild_sol_glue}, i.e. for $k=n-1$ and $t\in [(n-1)\tau,n\tau ]$ we have
\begin{align}
\mu_t(\dif x)=&\int\limits_{\mathbb{R}^d} p((n-1) \tau , x_0,t,x) \mu_{(n-1) \tau}(\dif x_0)\dif x\notag \\ &+\int\limits_{(n-1)\tau}^t \int\limits_{\mathbb{R}^d} p(s,x_0,t, x) \hat{\Lambda}(s,x_0) \mu_s(\dif x_0 )\dif s  \dif x   \notag \\
&+ \sum\limits_{j=1}^d \int\limits_{(n-1)\tau}^t \int\limits_{\mathbb{R}^d} \partial_{\widetilde{x_0},j} p(s,\widetilde{x_0},t, x) \hat{b}_j(s,\widetilde{x_0}) \mu_s(\dif \widetilde{x_0} )  \dif s  \dif x. \label{eq:measure_mild_glue_hypothesis_in_proof}
\end{align}
Now, plugging equation \eqref{eq:measure_mild_glue_induction_hypothesis}
into equation \eqref{eq:measure_mild_glue_hypothesis_in_proof}, we find
\begin{align}
&\mu_t(\dif x)\\
= &\int\limits_{\mathbb{R}^d} p((n-1) \tau , x_0,t,x) \Big (\int_{\mathbb{R}^d} p(0,\widetilde{x_0},(n-1)\tau, x_0) {\bf u_0} (\dif \widetilde{x_0}) \dif x_0  \notag\\
& +\int\limits_{0}^{(n-1)\tau} \int\limits_{\mathbb{R}^d} p(s,\widetilde{x_0},(n-1)\tau,x_0) \hat{\Lambda}(s,\widetilde{x_0}) \mu_s(\dif \widetilde{x_0} )  \dif s \dif x_0  \notag \\
&
 + \sum\limits_{j=1}^d \int\limits_{0}^{(n-1)\tau} \int\limits_{\mathbb{R}^d} \partial_{\widetilde{x_0},j} p(s,\widetilde{x_0},(n-1)\tau, x) \hat{b}_j(s,\widetilde{x_0}) \mu_s(\dif \widetilde{x_0} )  \dif s \dif x_0 \Big ) \dif x \notag \\
&+\int\limits_{(n-1)\tau}^t \int\limits_{\mathbb{R}^d} p(s,x_0,t, x) \hat{\Lambda}(s,x_0) \mu_s(\dif x_0) \dif s  \dif x \notag \\
&+  \sum\limits_{j=1}^d \int\limits_{(n-1)\tau}^t \int\limits_{\mathbb{R}^d} \partial_{x_0,j} p(s,x_0,t, x) \hat{b}_j(s,x_0) \mu_s(\dif x_0) \dif s \dif x\notag \\
= &\int\limits_{\mathbb{R}^d} \Big ( \int\limits_{\mathbb{R}^d} p(0,\widetilde{x_0},(n-1)\tau, x_0) p((n-1) \tau, x_0,t,x)\dif x_0 \Big ) {\bf u_0} (\dif \widetilde{x_0}) \dif x \notag\\
&+ \Big [ \int\limits_{0}^{(n-1)\tau}  \int\limits_{\mathbb{R}^d} \Big (\int\limits_{\mathbb{R}^d} p(s,\widetilde{x_0},(n-1)\tau,x_0) p((n-1) \tau , x_0,t,x)  \dif x_0 \Big ) \hat{\Lambda}(s,\widetilde{x_0}) 
\mu_s(\dif \widetilde{x_0}) \dif s \dif x    \notag\\
& + \int\limits_{(n-1)\tau}^t \int\limits_{\mathbb{R}^d} p(s,\widetilde{x_0},t, x) \hat{\Lambda}(s,\widetilde{x_0}) \mu_s(\dif \widetilde{x_0} ) \dif s  \dif x \Big ] \notag\\
&+ \Big [  \int\limits_{0}^{(n-1)\tau} \sum\limits_{j=1}^d \int\limits_{\mathbb{R}^d} \Big (\int\limits_{\mathbb{R}^d} \partial_{\widetilde{x_0},j} p(s,\widetilde{x_0},(n-1)\tau,x_0) p((n-1) \tau , x_0,t,x)  \dif x_0 \Big ) \hat{b}_j(s,\widetilde{x_0}) \mu_s(\dif \widetilde{x_0}) \dif s \dif x \notag \\
&+  \sum\limits_{j=1}^d \int\limits_{(n-1)\tau}^t \int\limits_{\mathbb{R}^d} \partial_{\widetilde{x_0},j} p(s,\widetilde{x_0},t, x) \hat{b}_j(s,\widetilde{x_0}) \mu_s(\dif \widetilde{x_0} ) \dif s  \dif x \Big ].
\label{eq:some_long_terms}
\end{align}
Now, focus on the right-hand side of the equation. For the first two terms, we can now use the Chapman-Kolmogorov identity (see \eqref{MarkovFund})
for $p$ which is included  in Assumption \ref{A4}. That is, for a.e.
$(t,x)$ and $(s,x_0)$, we have
\begin{align}
\int\limits_{\mathbb{R}^d} p(s,\widetilde{x_0},(n-1)\tau, x_0) p((n-1) \tau, x_0,t,x)\dif x_0 =p(s,\widetilde{x_0},t,x) \label{eq:chapman_kolmogorov}.
\end{align}
We differentiate \eqref{eq:chapman_kolmogorov} with
respect to $x_0$, we integrate 
the resulting identity
 against a test function and use Tonelli's theorem to conclude that
\begin{align}
\int\limits_{\mathbb{R}^d} \partial_{\widetilde{x_0},j}p(s,\widetilde{x_0},(n-1)\tau, x_0) p((n-1) \tau, x_0,t,x)\dif x_0 =\partial_{\widetilde{x_0},j} p(s,\widetilde{x_0},t,x) \label{eq:chapman_kolmogorov_diff}, \ (t,x),(s,x_0),   \ {\rm a.e.}
\end{align}
 for any $j \in \{1,\ldots,d\}$.
Applying \eqref{eq:chapman_kolmogorov} and \eqref{eq:chapman_kolmogorov_diff} to the right-hand side of \eqref{eq:some_long_terms} finally yields that $\mu$ is a measure-mild solution of \eqref{eq:linearized_PDE} in the sense of
 Definition 
\ref{def:measure_mild_solution}.
i.e. establishes $S_n$.
\item Now suppose that $\mu$ is a measure-mild solution of \eqref{eq:linearized_PDE} in the sense of Definition \ref{def:measure_mild_solution}.
\\ \\ 
We use a mathematical induction over $k \in \{0,1,\ldots,N-1\}$ to show that \eqref{eq:measure_mild_sol_glue} holds. Again, the induction basis is obvious because for $k=0$, \eqref{eq:measure_mild_sol_glue} and \eqref{eq:measure_mild_sol_def} co\"{\i}ncide. Now suppose that \eqref{eq:measure_mild_sol_glue} holds for all $k\in \{0,1,\ldots,n-1\}$. 
Recall the induction hypothesis for $k=(n-1)$ and $t =n\tau$:
\begin{align*}
\mu_{n\tau}(\dif x)&=\int\limits_{\mathbb{R}^d} p((n-1) \tau , x_0,n\tau, x) 
\mu_{(n-1) \tau}(\dif x_0)\dif x\\
& \quad +\int\limits_{(n-1)\tau}^{n\tau} \int\limits_{\mathbb{R}^d} p(s,x_0,n\tau ,x) \hat{\Lambda}(s,x_0) \mu_s(\dif x_0 ) \dif s \dif x  \\
& \hfill \quad   +  \sum\limits_{j=1}^d \int\limits_{(n-1)\tau}^{n\tau} \int\limits_{\mathbb{R}^d} \partial_{x_0} p(s,x_0,n\tau ,x) \hat{b}_j(s,x_0) \mu_s(\dif x_0 ) \dif s \dif x.
\end{align*}
We begin by exploiting the definition of a measure-mild solution of
 \eqref{eq:linearized_PDE} in the sense of Definition  \ref{def:measure_mild_solution}:
\begin{align*}
&\int\limits_{\mathbb{R}^d} p(n \tau , x_0,t, x) \mu_{n \tau}  (\dif x_0)\dif x
 +\int\limits_{n\tau}^t \int\limits_{\mathbb{R}^d} p(s,x_0,t,x) \hat{\Lambda}(s,x_0) \mu_s(\dif x_0 )  \dif s \dif x  \\
& \hfill \quad   +\sum\limits_{j=1}^d \int\limits_{n\tau}^t \int\limits_{\mathbb{R}^d} \partial_{x_0,j} p(s,x_0,t,x) \hat{b}_j(s,x_0) \mu_s(\dif x_0 )   \dif s \dif x\\
=&\int\limits_{\mathbb{R}^d} p(n \tau , x_0,t, x) \Big ( \int\limits_{\mathbb{R}^d} p(0, \widetilde{x_0},n\tau, x_0) \mu_0(\dif \widetilde{x_0})\dif x_0 \\
& \quad  +\int\limits_{0}^{n\tau} \int\limits_{\mathbb{R}^d} p(s,\widetilde{x_0},n\tau ,x_0) \hat{\Lambda}(s,\widetilde{x_0}) \mu_s(\dif \widetilde{x_0} ) \dif s \dif x_0  \\
& \hfill \quad    +  \sum\limits_{j=1}^d \int\limits_{0}^{n\tau} \int\limits_{\mathbb{R}^d} \partial_{\widetilde{x_0},j} p(s,\widetilde{x_0},n\tau ,x_0) \hat{b}_j(s,\widetilde{x_0}) \mu_s(\dif \widetilde{x_0} ) \dif s \dif x_0
 \Big ) \\
& \quad +\int\limits_{n\tau}^t \int\limits_{\mathbb{R}^d} p(s,x_0,t,x) \hat{\Lambda}(s,x_0) \mu_s(\dif x_0 )  \dif s \dif x  \\
& \hfill \quad   +\sum\limits_{j=1}^d \int\limits_{n\tau}^t \int\limits_{\mathbb{R}^d} \partial_{x_0,j} p(s,x_0,t,x) \hat{b}_j(s,x_0) \mu_s(\dif x_0 )   \dif s \dif x \\
=&\int\limits_{\mathbb{R}^d} \Big ( \int\limits_{\mathbb{R}^d} p(0 , \widetilde{x_0},n\tau, x_0) p(n \tau , x_0,t, x) \dif x_0 \Big ) \mu_0(\dif \widetilde{x_0})  \\
& \quad +\int\limits_{0}^{n\tau} \int\limits_{\mathbb{R}^d} p(s,\widetilde{x_0},n\tau ,x_0) \hat{\Lambda}(s,\widetilde{x_0}) \mu_s(\dif \widetilde{x_0} ) \dif s \dif x_0  \\
& \hfill \quad  +  \sum\limits_{j=1}^d \int\limits_{0}^{n\tau} \int\limits_{\mathbb{R}^d} \partial_{\widetilde{x_0},j} p(s,\widetilde{x_0},n\tau ,x_0) \hat{b}_j(s,\widetilde{x_0}) \mu_s(\dif \widetilde{x_0} ) \dif s \dif x_0 \\
& \quad +\int\limits_{n\tau}^t \int\limits_{\mathbb{R}^d} p(s,x_0,t,x) \hat{\Lambda}(s,x_0) \mu_s(\dif x_0 )  \dif s \dif x  \\
& \hfill \quad   + \sum\limits_{j=1}^d \int\limits_{n\tau}^t \int\limits_{\mathbb{R}^d} \partial_{x_0,j} p(s,x_0,t,x) \hat{b}_j(s,x_0) \mu_s(\dif x_0 )   \dif s \dif x \\
=& \int_{\mathbb{R}^d}p(0,\widetilde{x_0},t, x)\mu_0(\dif \widetilde{x_0}) +\int\limits_{0}^t \int\limits_{\mathbb{R}^d} p(s,x_0,t,x) \hat{\Lambda}(s,x_0) 
\mu_s(\dif x_0 )  \dif s \dif x  \\
& \hfill \quad   +\sum\limits_{j=1}^d \int\limits_{0}^t \int\limits_{\mathbb{R}^d} \partial_{x_0,j} p(s,x_0,t,x) \hat{b}_j(s,x_0) \mu_s(\dif x_0 )   \dif s \dif x \overbrace{=}^{\eqref{eq:measure_mild_sol_def}} \mu_t(\dif x ).
\end{align*}
Thus the claim of the inductive step is established and the proof accomplished.
\end{enumerate} \end{proof}

\noindent With these results, we can now formulate the central existence and uniqueness theorem. The idea will be to find local "mild solutions" by Lemma \ref{lem:fixed_point_prep} and then to "glue them together" by means of Lemma \ref{Lemma:glue}.
\begin{thm} \label{thm:e_and_u}
Let Assumptions 
\ref{A1}, \ref{A4},
 \ref{A6} and \ref{A7}
 be in force.
\begin{enumerate}
\item Under Assumption \ref{A5}, there exists a  unique
bounded mild solution of \eqref{eq:PDE_original}.
\item Under Assumption \ref{C5}, there is at most one bounded mild solution of \eqref{eq:PDE_original}.
\end{enumerate} 
 \end{thm}

\begin{proof}[Proof. \nopunct]
\begin{description}
\item {(a) Existence.}  
We assume Assumption \ref{A5}. To begin constructing the local "mild solutions", we set $r=0$, $\tau>0$
such that $T = N \tau$.
 We also set
 $\phi=u_0$ in \eqref{eq:def_u_zero_hat},  
where ${\bf u_0}(\dif x_0) =   u_0(x_0) \dif x_0$.
 Recall that we have assumed $u_0$ to be a bounded density, so both its $L^{\infty}$-norm and its $L^1$-norm can bounded by a constant $M\geq 1$. First, we define for $(t,x)\in [0,\tau]\times \mathbb{R}^d$
\begin{align}
\widehat{u_0}^0(t,x):= \hat u_0(0, u_0)(t,x) =
\int\limits_{\mathbb{R}^d}p(0,x_0,t,x) u_0(x_0)\dif x_0. \label{eq:u_0^0hat}
\end{align}
Now, we can apply Lemma \ref{lem:fixed_point_prep} to infer the existence of a fixed-point of the mapping 
$\Pi: L^1([0,\tau],L^1(\mathbb{R}^d)) \rightarrow  L^1([0,\tau],L^1(\mathbb{R}^d))$, defined by
\begin{align}
\Pi(r,v)(t,x):=&\int\limits_r^t \int\limits_{\mathbb{R}^d} p(s,x_0,t,x) 
(v+\widehat{u_0}(\phi))(s,x_0)
\Lambda(s,x_0,(v+\widehat{u_0})(\phi)(s,x_0)) 
\dif x_0  \dif s \notag \\
&+\sum\limits_{j=1}^d  \int\limits_r^t \int\limits_{\mathbb{R}^d} \partial_{x_0,j} p(s,x_0,t,x)
(v+\widehat{u_0}(\phi))(s,x_0)
 b_j(s,x_0,(v+\widehat{u_0}(\phi)(s,x_0))) \dif x_0  \dif s. \label{eq:def_pi_Second}
\end{align}
We will refer to this fixed-point as $v^0$. Note that $v^0 \in L^1([0,\tau],B(0,M))\cap B_{\infty}(0,M)$. Then, we set $u^0(t,x)=\widehat{u_0}(t,x)+v^0(t,x)$ for 
$(t,x)\in [0,\tau]\times \mathbb{R}^d$. Note that $u^0$ satisfies equation \eqref{eq:mild_sol_def} from the definition of mild solution restricted on $t\in [0,\tau]$. \\ \\
To extend the construction of $u^0$ for values of $t$ above $\tau$, we proceed by induction. Fix some 
$k \in \{1,\ldots,N-1\}$ and suppose we are given a family of functions $u^1,u^2,\ldots,u^{k-1}$ such that for any $\ell\in \{1,\ldots,k-1\}$ it holds that $u^\ell \in L^1([\ell \tau, (\ell + 1)\tau],L^1 (\mathbb{R}^d)) \cap L^{\infty}([\ell\tau, (\ell + 1)\tau] \times \mathbb{R}^d,\mathbb{R})$ and $u^\ell$ satisfies 
\begin{align} \label{eq:local_mild_solutions} 
u^\ell(t,x)=&\int\limits_{\mathbb{R}^d} p(\ell \tau, x_0,t, x) u^{\ell-1}  (\ell \tau , x_0) \dif x  \nonumber \\
& +\int\limits_{\ell\tau}^t \int\limits_{\mathbb{R}^d} p(s,x_0,t,x) 
\Lambda(s,x_0,u^\ell(s,x_0))
u^\ell (s,x_0) 
 \dif x_0 \dif s \dif x  \\
& + \sum\limits_{j=1}^d \int\limits_{\ell\tau}^t \int\limits_{\mathbb{R}^d} \partial_{x_0,j}
 p(s,x_0,t,x) 
b_j (s,x_0,u^\ell(s,x_0))
u^\ell (s,x_0)  \dif x_0 \dif s \dif x, \nonumber
\end{align}
for all $(t,x)\in [\ell\tau, (\ell + 1)\tau]\times \mathbb{R}^d$. 
In order to define $u^k$, we begin by defining $\widehat{u_0}^k $
\begin{align*}
\widehat{u_0}^k (t,x):= \hat u_0(k \tau, u^{k-1}(k \tau,\cdot))(t,x)
=\int\limits_{\mathbb{R}^d} p(k\tau,x_0,t,x)u^{k-1}(k\tau,x_0) \dif x_0,
\end{align*}
for $(t,x)\in [k\tau,(k+1)\tau]$. Again, we can choose $M$ large enough so as to satisfy $M \geq \max\{ \left\Vert u_{k-1} (k\tau, \cdot)\right\Vert_{\infty},\left\Vert u_{k-1}(k\tau, \cdot)\right\Vert_1\}$. Then set $r = k\tau$ and 
$\phi= u^{k-1}(k\tau,\cdot)$ which allows us to apply Lemma \ref{lem:fixed_point_prep} thereby establishing existence and 
uniqueness of a function $v^k : [k\tau, (k+1)\tau]\times \mathbb{R}^d \rightarrow \mathbb{R}^d$ that
belongs to $L^1([k\tau,(k + 1)\tau],B(0,M)) \cap  B_{\infty}(0,M)$ and satisfying
\begin{align*}
v^k(t,x)= & 
\int\limits_{k\tau}^t \int\limits_{\mathbb{R}^d} p(s,x_0,t,x) 
\Lambda(s,x_0,u^k(s,x_0))
u^k (s,x_0 )
  \dif x_0 \dif s \dif x  \\
& +\sum\limits_{j=1}^d \int\limits_{k\tau}^t \int\limits_{\mathbb{R}^d} \partial_{x_0,j} p(s,x_0,t,x) b_j(s,x_0,u^k(s,x_0)) u^k (s,x_0 )  \dif x_0 \dif s \dif x,
\end{align*}
for all $(t,x)\in [k\tau, (k + 1)\tau]\times \mathbb{R}^d$. Again, we define $u^k(t,x):=\widehat{u_0^k}(t,x)+v^k(t,x)$ for $(t,x)\in [k\tau,(k+1)\tau] \times \R^d$. In particular, for $(t,x)\in [k\tau,(k+1)\tau] \times \R^d$ we have
\begin{align*}
u^k(t,x)= &\int\limits_{\mathbb{R}^d} p(k\tau,x_0,t,x)u^{k-1}(k\tau,x_0)\dif x \\
& +\int\limits_{k\tau}^t \int\limits_{\mathbb{R}^d} p(s,x_0,t,x) 
\Lambda(s,x_0,u^k (s,x_0 ))
 u^k (s,x_0 )  \dif x_0 \dif s \dif x  \\
& +\sum\limits_{j=1}^d \int\limits_{k\tau}^t \int\limits_{\mathbb{R}^d} \partial_{x_0,j} p(s,x_0,t,x)
 b_j(s,x_0, u^k (s,x_0 ) )
u^k (s,x_0 )  \dif x_0 \dif s \dif x.
\end{align*}
This shows that we can indeed use induction to define a sequence $u^1,u^2,\ldots , u^N$ of  functions such that for any $\ell \in \{1,...,N-1\}$ we have $u^\ell \in L^1([\ell\tau, (\ell + 1)\tau],L^1 (\mathbb{R}^d)) \cap L^{\infty}([\ell\tau, (\ell + 1)\tau] \times \mathbb{R}^d,\mathbb{R})$ which satisfies \eqref{eq:local_mild_solutions} for any
 $(t,x)\in [\ell\tau,(\ell+1)\tau] \times \R^d$. \\ \\
Thus we can define a function $u$ on $[0,T]\times \mathbb{R}^d$ by setting $u(t,x)=u^\ell(t,x)$ if $t\in [\ell\tau,(\ell+1)\tau]$. Invoking Lemma \ref{Lemma:glue} with
$$\hat \Lambda(s,x) = \Lambda(s,x,u(s,x)), \ \hat b(s,x) = b(s,x,u(s,x)), 
\
\mu_t(\dif x):=u(t,x)\dif x,$$
 shows that $\mu$ is a measure-mild solution of \eqref{eq:linearized_PDE}. Recalling Remark \ref{rem:measure_mild_def_test} we see that $u$ is a mild solution of \eqref{eq:PDE_original}.  
\item {(b) Uniqueness.}
Here we only suppose the validity of Assumption \ref{C5}.
To establish uniqueness of a bounded mild solution, suppose that $u_1$ and $u_2$ are two bounded mild solutions of \eqref{eq:PDE_original} in the sense
of Definition \ref{def_mild_sol}. 
Then with similar estimations as 
for proving \eqref{E33bis},
we find that there exists a constant $C$, depending only on $u_1,u_2,\Lambda,\Phi,b$ and $b_0$ such that for all $\ell\in \{0, \ldots, N-1\}$, we have 
\begin{align*}
\int\limits_{0}^T \left\Vert u_1(t,\cdot)-u_2(t,\cdot) \right\Vert_{L^1} \dif t \leq \frac{C^\ell}{\ell!} \int\limits_{0}^T \left\Vert u_1(t,\cdot)-u_2(t,\cdot) \right\Vert_{L^1} \dif t,
\end{align*}
from which we can conclude that $ \int\limits_{0}^T \left\Vert u_1(t,\cdot)-u_2(t,\cdot) \right\Vert_{L^1} \dif t=0$. 
\end{description}
\end{proof}

\section{Uniqueness 
for the linear PDE}\label{subsec_equiv_3}

We fix the same framework as the one of Section \ref{SMeasureMild}.

\subsection{Uniqueness of measure-mild solutions}

\begin{prop} \label{prop:uniqueness_measure_mild}
Assume \ref{A4} and that $\hat{b} :[0,T]\times \R^d \rightarrow \R^d $ and $\hat{\Lambda}:[0,T]\times \R^d \rightarrow \R$ are measurable and bounded.
Then there is at most one measure-mild solution of \eqref{eq:linearized_PDE}. 
\end{prop}
\begin{proof}[Proof. \nopunct]
Suppose that $\mu_1$ and $\mu_2$ are measure-mild solutions of \eqref{eq:linearized_PDE}. Then we define $\nu:=\mu_1-\mu_2$. Note that 
the definition of measure-mild solutions implies that $\nu$ satisfies
\begin{align*}
\nu_t(\dif x)= &\int\limits_0^t \int\limits_{\mathbb{R}^d}\hat{\Lambda}(r,x_0) p(r,x_0,t,x)\dif x 
~\nu_r(\dif x_0) \dif r \\
&+ \sum\limits_{j=1}^d \int\limits_0^t \int\limits_{\mathbb{R}^d} \hat{b}_j(r,x_0) \partial_{x_0,j} p(r,x_0,t, x)\dif x ~ \nu_r(\dif x_0) \dif r. 
\end{align*}
We prepare the application of  Gronwall's Lemma estimating  the total variation norm of $\nu_t(\cdot)$.
\\ \\
\textbf{Step 1. Estimation of the total variation norm of} $\mathbf{\nu}$ \\ \\ 
Now fix an arbitrary $t\in [0,T]$. Then we have 
\begin{align*}
\Vert \nu_t(\cdot) \Vert_{TV}  =&  
\sup\limits_{\substack{\Psi \in C_b(\mathbb{R}^d)\\ \Vert \Psi \Vert_{\infty} \leq 1 }} 
\Big \vert \int\limits_0^t \int\limits_{\mathbb{R}^d}\hat{\Lambda}(r,x_0) 
\Big( \int\limits_{\mathbb{R}^d} \Psi(x) p(r,x_0,t,x)\dif x\Big)  \nu_r(\dif x_0) \dif r   \\
&+    \sum\limits_{j=1}^d \int\limits_0^t \int\limits_{\mathbb{R}^d} \hat{b}_j(r,x_0) \Big( \int\limits_{\mathbb{R}^d} \Psi(x) \partial_{x_0,j} p(r,x_0,t,x) \dif x \Big) \nu_r(\dif x_0) \dif r  \Big \vert \\
\leq &\sup\limits_{\substack{\Psi \in C_b(\mathbb{R}^d)\\ \left\Vert \Psi \right\Vert_{\infty} \leq 1 }} \Big\vert \int\limits_{\mathbb{R}^d} \int\limits_0^t \hat{\Lambda}(r,x_0) \Big ( \int\limits_{\mathbb{R}^d} \Psi(x)p(r,x_0,t,x)\dif x \Big) \nu_r(\dif x_0) \dif r \Big\vert \\
&+ \sup\limits_{\substack{\Psi \in C_b(\mathbb{R}^d)\\ \Vert \Psi \Vert_{\infty} \leq 1 }} \sum\limits_{j=1}^d \Big \vert \int\limits_{\mathbb{R}^d} \int\limits_0^t \hat{b}_j(r,x_0) \Big ( \int\limits_{\mathbb{R}^d} \Psi(x)\partial_{x_0,j} p(r,x_0,t,x)\dif x\Big ) \nu_r(\dif x_0) \dif r  \Big\vert \\
\leq &\sup\limits_{\substack{\Psi \in C_b(\mathbb{R}^d)\\ \Vert \Psi \Vert_{\infty} \leq 1 }} \int\limits_{\mathbb{R}^d} \int\limits_0^t  \vert \hat{\Lambda}(r,x_0)  \vert \Big ( \int\limits_{\mathbb{R}^d} \underbrace{\vert \Psi(x) \vert}_{\leq 1} p(r,x_0,t,x)\dif x \Big ) \vert \nu_r \vert (\dif x_0)  \dif r  \\
&+ \sup\limits_{\substack{\Psi \in C_b(\mathbb{R}^d)\\ \Vert \Psi \Vert_{\infty} \leq 1 }} \sum\limits_{j=1}^d \int\limits_{\mathbb{R}^d} \int\limits_0^t \vert \hat{b}_j(r,x_0) \vert \Big( \int\limits_{\mathbb{R}^d} \underbrace{\vert\Psi(x) \vert}_{\leq 1} \underbrace{\vert \partial_{x_0,j} p(r,x_0,t,x) \vert}_{\leq C_u \frac{1}{\sqrt{t-r}} q(r,x_0,t,x) \ \ \eqref{eq:1011} } \dif x \Big ) \vert \nu_r(\dif x_0)  \dif r  \\
\leq & \int\limits_{\mathbb{R}^d} \int\limits_0^t \underbrace{\vert  \hat{\Lambda}(r,x_0)\vert }_{\leq M_{\Lambda}}  \underbrace{\int\limits_{\mathbb{R}^d} p(r,x_0,t,x)\dif x}_{=1} \vert \nu_r \vert (\dif x_0)  \dif r  \\
&+C_u \int\limits_{\mathbb{R}^d} \int\limits_0^t \underbrace{\vert  \hat{b}_j(r,x_0)\vert }_{\leq M_{b}} \frac{1}{\sqrt{t-r}} \underbrace{\int\limits_{\mathbb{R}^d}q(r,x_0,t,x) \dif x}_{= 1} 
\vert \nu_r \vert (\dif x_0)  \dif r \\
\leq & C_1 \int\limits_0^t \Big(1+\frac{1}{\sqrt{t-r}}\Big)\int\limits_{\mathbb{R}^d} \vert \nu_r 
\vert (\dif x_0) \dif r \\
=& C_1 \int\limits_0^t \Big(1+\frac{1}{\sqrt{t-r}}\Big) \Vert \nu_r \Vert_{TV} \dif r,\\
\end{align*}
where we have used inequality \eqref{eq:1011}, the fact that $x \mapsto p(s,x_0,t,x)$
 is a probability density and the definition of the total variation norm. For future reference, we note that we have shown
\begin{eqnarray*}
\Vert \nu_t \Vert_{TV} &\leq& C_1 \int\limits_0^t \Big(1+\frac{1}{\sqrt{t-s}}\Big) \Vert 
\nu_s \Vert_{TV} \dif s, \\
&\leq&  C \int\limits_0^t \frac{1}{\sqrt{t-s}} \Vert \nu_s\Vert_{TV} \dif s, \label{eq:unfortunately_not_Gronwall}
\end{eqnarray*}
for $C:= C_1 (T+1)$.

\noindent \textbf{Step 2. Iterating the estimation and Gronwall lemma} \\ \\
Similarly to the estimates from  \eqref{RFTrick} to
\eqref{RFTrickBis}, \eqref{eq:unfortunately_not_Gronwall} allows to prove
 \begin{align}
 \Vert \nu_t \Vert_{TV} \leq \pi C^2  \int\limits_0^t \Vert   \nu_r  \Vert_{TV} \dif r. \label{eq:Gronwall_indeed}
 \end{align}
Applying Grownwall's Lemma on the basis of \eqref{eq:Gronwall_indeed}, we infer that 
$\Vert \nu_t \Vert_{TV}=0$. 
Because this holds for any $t\in [0,T]$ and since $\nu = \mu_1 - \mu_2$, this is exactly the claimed uniqueness of measure-mild solutions. \end{proof}

\subsection{Equivalence weak-mild for the linear PDE}
Similarly as for the nonlinear PDE where we used the notion of weak
solution, see Definition \ref{def_weak_sol}, we  shall also need to introduce the concept of measure-weak solutions
of \eqref{eq:linearized_PDE}.

\begin{defi} \label{def:measure_weak_solution}
 Assume \ref{A1} and that $\hat{b}$ and $\hat{\Lambda}$ are bounded. We then call $\mu :[0,T] \rightarrow \mathcal{M}_f(\mathbb{R}^d)$ a \textbf{measure-weak} solution of \eqref{eq:linearized_PDE} if we have for all $\varphi\in C_0^{\infty}$ 
\begin{align*}
\int\limits_{\mathbb{R}^d} \varphi(x) \mu_t(\dif x)  =&\int\limits_{\mathbb{R}^d} \varphi(x) \mathbf{u_0}(\dif x) +\sum\limits_{j=1}^d \int\limits_{\mathbb{R}^d} \int\limits_0^t \partial_j \varphi (x) \hat{b}(s,x) \mu_s(\dif x) \dif s \\
&+ \int\limits_0^t \int\limits_{\mathbb{R}^d} L_s \varphi(x) \mu_s(\dif x) \dif s+\int\limits_0^t \int\limits_{\mathbb{R}^d} \varphi(x)	\hat{\Lambda}(s,x)\mu_s(\dif x) \dif s.
\end{align*}
\end{defi}

\noindent We can naturally adapt Proposition \ref{prop:equiv_mild_weak_sol} to the linearized case.
\begin{prop}\label{prop:equiv_mild_weak_sol_lin} 
We assume \ref{A1}, \ref{A4}, \ref{C} and that $\hat{b}, \hat{\Lambda}$ are bounded.
Let $ \mu:[0,T] \rightarrow \mathcal{M}_f(\mathbb{R}^d)$. 
$\mu$ is a measure-mild solution of \eqref{eq:linearized_PDE}  if and only if $\mu$ is a measure-weak solution.
\end{prop}
\begin{proof}[Proof \nopunct.]
The proof is analogous to the proof of Proposition \ref{prop:equiv_mild_weak_sol}  if one keeps the second statement of Remark \ref{rem:measure_mild_def_test} in mind. 
\end{proof}

%
%

\section{Appendix}
\subsection{A sufficient condition for Assumption 
\ref{A4}}\label{sec:Friedman}

\begin{prop} \label{PAppReg}
Let us suppose the following properties.
\begin{enumerate}
\item For all $i,j \in \{1,\ldots,d\}$ for all $t\in [0,T]$,
 the mapping $a_{ij}(t,\cdot)$ is in $C^2$ and  
the mapping $b_0(t,\cdot)$ is in $C^1$
whose  derivatives
with respect to the space variable $x$ are bounded on 
$[0,T] \times {\mathbb R}^d$.
\item The partial derivatives of order 2 (resp. order 1) for the components of $a$ (resp. $b_0$)  
are H\"older continuous in space, uniformly with respect to time, with some parameter $\alpha$.
\item Assumptions \ref{A1} and \ref{A2} hold.
\end{enumerate}
Then Assumption \ref{A4} is verified. 
In particular
there is a Markov fundamental solution 
of \eqref{eq:rather_implicit}, in the sense of  
 Definition \ref{D3}.
\end{prop}
\begin{proof}[\nopunct Proof.] 
In  the definition stated in sect. 1, p.3 of \cite{friedmanEDP})
 Friedman considers a notion of fundamental solution of
\begin{eqnarray} 
\partial_t u&=&L_t^*u, \label{Lp} \\
\partial_tu &=& L_tu. \label{Lgamma}
\end{eqnarray}
By Theorems 15, section 9, chap. 1 in \cite{friedmanEDP}
 and inequalities (8.13) and (8.14) just before,
there exist fundamental solutions $p, \Gamma$ in the sense of Friedman
 of \eqref{Lp}, \eqref{Lgamma} fulfilling 
\eqref{eq:1011a} and \eqref{eq:1011}
such that
 \begin{equation}
 \label{eq:Gammap}
 p(s,x_0,t,x)=\Gamma(T-t,x,t-s,x_0)\ .
 \end{equation}
We can verify then that $p$ is a fundamental solution 
 of \eqref{eq:FokkerPlanck} in the sense of Definition \ref{D3}.
The basic argument for this consists essentially in the fact that a smooth
solution of a PDE is also a solution in the sense of distributions;
this establishes \eqref{eq:rather_implicit}. 
By density arguments it is clear that  \eqref{eq:rather_implicit}
also extends to $\varphi \in C_b$.
Taking there $\varphi = 1$ 
for every $\nu_0$ of Delta function type, we get 
\eqref{Intequal1}.
\\
An important point is constituted by the fact that 
this fundamental solution is also a transition probability, i.e.
it verifies \eqref{MarkovFund}:
this is the object of Problem 5, Chapter 1  of \cite{friedmanEDP}.

\end{proof}

\subsection{Existence and uniqueness of a solution to a classical SDE}
\label{subsec_equiv_2}
In this section we emphasize that we will state our results
in two different contexts: in law, using Stroock-Varadhan
 (\cite{stroock}) or Krylov (\cite{krylov})
 arguments for well-posedness of weak solutions of classical SDEs or Veretennikov (\cite{veretennikov1981strong})
  for  strong solutions. 
\\
We fix a measurable function $u:[0,T]\times \mathbb{R}^d \rightarrow \R$.
\begin{lem}
\label{T127}
Let $\Phi: [0,T] \times \R^d \rightarrow \R^{d\times d}$ be such that $a(t,x) = \Phi(t,x) \Phi(t,x)^t$ for all  $(t,x) \in [0,T] \times \R^d.$ 
Let  $\nu$ be a Borel probability measure on $\mathbb R^d$. 
Suppose the validity of Assumptions \ref{A1},~\ref{A2},~\ref{A3}
and that $(t,x) \mapsto b(t,x,u(t,x))$ and 
$(t,x) \mapsto \Lambda(t,x,u(t,x))$ 
are bounded. Then
the SDE
\begin{align}\label{eq:FP_SDE_initial}
\begin{cases}
\dif X_t &= \Phi(t,X_t)\dif B_t + (b_0(t,X_t)+b(t,X_t,u(t,X_t))\dif t, \\
X_0 &\sim {\bf u_0}.
\end{cases}
\end{align}
where $ {\bf u}_0$
 is a probability measure admits weak existence and uniqueness 
in law ($B$ being a Brownian motion). 
\end{lem}
\begin{proof}
We consider the martingale problem associated to 
\begin{align*}
L_t=\frac{1}{2}\sum\limits_{i,j=1}^d a_{ij}(t,x) \partial_{ij}^2 + \sum\limits_{j=1}^d (b_{0,j}(t,x)+b_j(t,x,u(t,x))) \partial_j, 
\end{align*}
as introduced by Stroock and Varadhan in Chapter 6, Section 0 of \cite{stroock}.
We then note that Assumptions \ref{A1}, \ref{A2} and \ref{A3} are exactly the ones required for Theorem 7.2.1 in Chapter 7, Section 2 of \cite{stroock}, according to which our martingale problem is well-posed. 
We can then apply Corollary 3.4 (alongside with Proposition 3.5) in Chapter 5, Section 3 of \cite{ethier}, yielding that SDE \eqref{eq:FP_SDE_initial} admits weak existence and uniqueness in law. 
\end{proof}
\noindent 
\begin{rem}\label{rem:e_without_u}
Note that \ref{A3} is only required for the uniqueness in law of \eqref{eq:FP_SDE_initial}. If we drop \ref{A3} but assume instead that $\Phi$ is non-degenerate, we still have weak existence of a solution to \eqref{eq:FP_SDE_initial}, as can be inferred from Theorem 1 in Chapter 2.6.
in \cite{krylov} and the same arguments as in the proof of
Lemma \ref{T127}. 
Indeed 
it is possible to show that if $a$ is nondegenerate, it is always possible to find a $\Phi$ such that $a=\Phi \Phi^t$ where $\Phi$ is nondegenerate.
\end{rem}
\noindent Under more restrictive assumptions, we have strong existence and pathwise uniqueness of a solution to \eqref{eq:FP_SDE_initial}.
\begin{lem} \label{veretennikov}
Assume \ref{A1}, \ref{A2}, \ref{B3} and that $(t,x) \mapsto b(t,x,u(t,x))$ and 
$(t,x) \mapsto \Lambda(t,x,u(t,x))$ 
are bounded.
 Then \eqref{eq:FP_SDE_initial} 
admits strong existence and pathwise uniqueness.
\end{lem}
\begin{proof}[Proof \nopunct.]
The result follows from Theorem 1 in \cite{veretennikov1981strong}, because $b$ is assumed to be bounded and $u$ is assumed to be measurable.

\end{proof}
\begin{rem} \hfill
\begin{itemize}
\item If $d=1$, 
Theorem 1 in \cite{veretennikov1981strong},
 admits some extensions, see Theorem 2 in  \cite{veretennikov1981strong}. 
\item If $\Phi=I_d$ 
 \cite{veretennikov1981strong} result
admits extensions to unbounded drifts, see  \cite{KrylovRoeck}.
\end{itemize}
\end{rem}

{\bf ACKNOWLEDGEMENTS.} 
The research related to this article
was supported by a public grant as part of the
{\it Investissement d'avenir project, reference ANR-11-LABX-0056-LMH,
  LabEx LMH,}
in a joint call with Gaspard Monge Program for optimization, 
operations research and their interactions with data sciences.
Partial financial support of the third named author was  provided
   by the DFG through the CRC ''Taming uncertainty and profiting from 
  randomness and low regularity in analysis, stochastics and their application''.

\bibliographystyle{plain}
\bibliography{NonConservativePDEJonas}

\begin{thebibliography}{10}

\bibitem{aronsonb}
D.~G. Aronson and H.~F. Weinberger.
\newblock Multidimensional nonlinear diffusion arising in population genetics.
\newblock {\em Adv. Math.}, 30:33--76, 1978.

\bibitem{BRR2}
V.~Barbu, M.~R\"ockner, and F.~Russo.
\newblock Probabilistic representation for solutions of an irregular porous
  media type equation: the irregular degenerate case.
\newblock {\em Probab. Theory Related Fields}, 151(1-2):1--43, 2011.

\bibitem{BRR3}
V.~Barbu, M.~R{\"o}ckner, and F.~Russo.
\newblock Doubly probabilistic representation for the stochastic porous media
  type equation.
\newblock {\em Ann. Inst. Henri Poincar\'e Probab. Stat.}, 53(4):2043--2073,
  2017.

\bibitem{BCR3}
N.~Belaribi, F.~Cuvelier, and F.~Russo.
\newblock Probabilistic and deterministic algorithms for space multidimensional
  irregular porous media equation.
\newblock {\em SPDEs: Analysis and Computations}, 1(1):3--62, 2013.

\bibitem{BCR2}
N.~Belaribi and F.~Russo.
\newblock Uniqueness for {F}okker-{P}lanck equations with measurable
  coefficients and applications to the fast diffusion equation.
\newblock {\em Electron. J. Probab.}, 17:no. 84, 28, 2012.

\bibitem{Ben_Vallois}
S.~Benachour, P.~Chassaing, B.~Roynette, and P.~Vallois.
\newblock Processus associ\'es \`a\ l'\'equation des milieux poreux.
\newblock {\em Ann. Scuola Norm. Sup. Pisa Cl. Sci. (4)}, 23(4):793--832, 1996.

\bibitem{BRR1}
P.~Blanchard, M.~R{\"o}ckner, and F.~Russo.
\newblock Probabilistic representation for solutions of an irregular porous
  media type equation.
\newblock {\em Ann. Probab.}, 38(5):1870--1900, 2010.

\bibitem{bogachevSFB}
V.~I. Bogachev, M.~R\"ockner, and S.~V. Shaposhnikov.
\newblock On uniqueness of solutions to the {C}auchy problem for degenerate
  {F}okker-{P}lanck-{K}olmogorov equations.
\newblock {\em J. Evol. Equ.}, 13(3):577--593, 2013.

\bibitem{bossytalay1}
M.~Bossy and D.~Talay.
\newblock Convergence rate for the approximation of the limit law of weakly
  interacting particles: application to the {B}urgers equation.
\newblock {\em Ann. Appl. Probab.}, 6(3):818--861, 1996.

\bibitem{DelarueMenozzi}
F.~Delarue and S.~Menozzi.
\newblock An interpolated stochastic algorithm for quasi-linear {PDE}s.
\newblock {\em Math. Comp.}, 77(261):125--158 (electronic), 2008.

\bibitem{ethier}
S.~N. Ethier and T.~G. Kurtz.
\newblock {\em Markov Processes Characterization and Convergence}.
\newblock Wiley Series in probability and statistics. John Wiley \& Sons, 1986.

\bibitem{figalli}
A.~Figalli.
\newblock Existence and uniqueness of martingale solutions for {SDE}s with
  rough or degenerate coefficients.
\newblock {\em J. Funct. Anal.}, 254(1):109--153, 2008.

\bibitem{friedmanEDP}
A.~Friedman.
\newblock {\em Partial differential equations of parabolic type}.
\newblock Prentice-Hall, Inc., Englewood Cliffs, N.J., 1964.

\bibitem{JourMeleard}
B.~Jourdain and S.~M{\'e}l{\'e}ard.
\newblock Propagation of chaos and fluctuations for a moderate model with
  smooth initial data.
\newblock {\em Ann. Inst. H. Poincar\'e Probab. Statist.}, 34(6):727--766,
  1998.

\bibitem{kac}
M.~Kac.
\newblock {\em Probability and related topics in physical sciences}, volume
  1957 of {\em With special lectures by G. E. Uhlenbeck, A. R. Hibbs, and B.
  van der Pol. Lectures in Applied Mathematics. Proceedings of the Summer
  Seminar, Boulder, Colo.}
\newblock Interscience Publishers, London-New York, 1959.

\bibitem{keener}
J.~P. Keener and J.~Sneyd.
\newblock {\em Mathematical Physiology II: Systems Physiology}.
\newblock Springer, New York, 2008.

\bibitem{krylov}
N.~V. Krylov.
\newblock {\em Controlled diffusion processes}, volume~14 of {\em Stochastic
  Modelling and Applied Probability}.
\newblock Springer-Verlag, Berlin, 2009.
\newblock Translated from the 1977 Russian original by A. B. Aries, Reprint of
  the 1980 edition.

\bibitem{KrylovRoeck}
N.~V. Krylov and M.~R\"ockner.
\newblock Strong solutions of stochastic equations with singular time dependent
  drift.
\newblock {\em Probab. Theory Related Fields}, 131(2):154--196, 2005.

\bibitem{LOR2}
A.~Le~Cavil, N.~Oudjane, and F.~Russo.
\newblock Particle system algorithm and chaos propagation related to a
  non-conservative {M}c{K}ean type stochastic differential equations.
\newblock {\em Stochastics and Partial Differential Equations: Analysis and
  Computation.}, pages 1--37, 2016.

\bibitem{LOR1}
A.~Le~Cavil, N.~Oudjane, and F.~Russo.
\newblock Probabilistic representation of a class of non-conservative nonlinear
  partial differential equations.
\newblock {\em ALEA Lat. Am. J. Probab. Math. Stat}, 13(2):1189--1233, 2016.

\bibitem{LOR4}
A.~Le~Cavil, N.~Oudjane, and F.~Russo.
\newblock Monte-{C}arlo algorithms for a forward {F}eynman--{K}ac-type
  representation for semilinear nonconservative partial differential equations.
\newblock {\em Monte Carlo Methods Appl.}, 24(1):55--70, 2018.

\bibitem{LOR3}
A.~Le~Cavil, N.~Oudjane, and F.~Russo.
\newblock Forward {F}eynman-{K}ac type representation for semilinear
  nonconservative partial differential equations.
\newblock {\em Stochastics: an International Journal of Probability and
  Stochastic Processes}, to appear. First version 2016, Preprint hal-01353757.

\bibitem{Mckeana}
H.~P.~Jr. McKean.
\newblock A class of {M}arkov processes associated with nonlinear parabolic
  equations.
\newblock In {\em Proc. {N}at. {A}cad. {S}ci. {U.S.A}., 1966)}, pages
  1907--1911. 1966.

\bibitem{Mckean}
H.~P.~Jr. McKean.
\newblock Propagation of chaos for a class of non-linear parabolic equations.
\newblock In {\em Stochastic {D}ifferential {E}quations ({L}ecture {S}eries in
  {D}ifferential {E}quations, {S}ession 7, {C}atholic {U}niv., 1967)}, pages
  41--57. Air Force Office Sci. Res., Arlington, Va., 1967.

\bibitem{MeleaRoel}
S.~M{\'e}l{\'e}ard and S.~Roelly-Coppoletta.
\newblock A propagation of chaos result for a system of particles with moderate
  interaction.
\newblock {\em Stochastic Process. Appl.}, 26(2):317--332, 1987.

\bibitem{murray}
J.~D. Murray.
\newblock {\em Mathematical biology. {I}}, volume~17 of {\em Interdisciplinary
  Applied Mathematics}.
\newblock Springer-Verlag, New York, third edition, 2002.
\newblock An introduction.

\bibitem{roeckner2010}
M.~R{\"o}ckner and X.~Zhang.
\newblock Weak uniqueness of {F}okker-{P}lanck equations with degenerate and
  bounded coefficients.
\newblock {\em C. R. Math. Acad. Sci. Paris}, 348(7-8):435--438, 2010.

\bibitem{stroock}
D.~W. Stroock and S.~R.~S. Varadhan.
\newblock {\em Multidimensional diffusion processes}.
\newblock Classics in Mathematics. Springer-Verlag, Berlin, 2006.
\newblock Reprint of the 1997 edition.

\bibitem{sznitman}
A-S. Sznitman.
\newblock Topics in propagation of chaos.
\newblock In {\em \'{E}cole d'\'{E}t\'e de {P}robabilit\'es de {S}aint-{F}lour
  {XIX}---1989}, volume 1464 of {\em Lecture Notes in Math.}, pages 165--251.
  Springer, Berlin, 1991.

\bibitem{veretennikov1981strong}
A.~Ju. Veretennikov.
\newblock On strong solutions and explicit formulas for solutions of stochastic
  integral equations.
\newblock {\em Sbornik: Mathematics}, 39(3):387--403, 1981.

\bibitem{walsh1986introduction}
J.~Walsh.
\newblock An introduction to stochastic partial differential equations.
\newblock {\em {\'E}cole d'{\'E}t{\'e} de Probabilit{\'e}s de Saint Flour
  XIV-1984}, pages 265--439, 1986.

\bibitem{wang}
X.~Y. Wang, Z.~S. Zhu, and Y.~K. Lu.
\newblock Solitary wave solutions of the generalized {B}urgers-{H}uxley
  equation.
\newblock {\em J. Phys. A: Math. Gen.}, 23:271--274, 1990.

\end{thebibliography}
\end{document}